\documentclass[a4paper,11pt]{article}
\usepackage{array}
\usepackage{amsmath}
\usepackage{amssymb}
\usepackage{amsthm}
\usepackage{graphicx}
\usepackage{enumerate}
\usepackage[dvipsnames]{xcolor}
\usepackage{hyperref}
\usepackage{subcaption}
\usepackage{comment}
\newtheorem{theorem}{Theorem}[section]

\newtheorem{lemma}[theorem]{Lemma}
\newtheorem{corollary}[theorem]{Corollary}

\newtheorem{assumption}{Assumption A}
\usepackage{cite}
\theoremstyle{definition}

\numberwithin{equation}{section}
\newcommand{\N}{\mathcal{N}}
\newcommand{\D}{\mathcal{D}}

\newcommand{\A}{\mathcal{A}}

\begin{document}
	
\title{IPAS: An Adaptive Sample Size Method for Weighted Finite Sum Problems with Linear Equality Constraints }

\author{Nata\v{s}a Kreji\'c\footnote{Department of Mathematics and Informatics, Faculty of Sciences, University of Novi Sad, Trg Dositeja Obradovi\' ca 4, 21000 Novi Sad, Serbia. e-mail: \texttt{natasak@uns.ac.rs}}, Nata\v{s}a Krklec Jerinki\'c \footnote{Department of Mathematics and Informatics, Faculty of Sciences, University of Novi Sad, Trg Dositeja Obradovi\' ca 4, 21000 Novi Sad, Serbia. e-mail: \texttt{natasa.krklec@dmi.uns.ac.rs}}, Sanja Rapaji\'c\footnote{Department of Mathematics and Informatics, Faculty of Sciences, University of Novi Sad, Trg Dositeja Obradovi\' ca 4, 21000 Novi Sad, Serbia. e-mail: \texttt{sanja@dmi.uns.ac.rs}},
Luka Rute\v{s}i\'c \footnote{Department of Mathematics and Informatics, Faculty of Sciences, University of Novi Sad, Trg Dositeja Obradovi\' ca 4, 21000 Novi Sad, Serbia. e-mail: \texttt{luka.rutesic@dmi.uns.ac.rs}} \footnote{Corresponding author}}  
\date{}
\maketitle

\begin{abstract}
Optimization problems with the objective function  in the form of weighted sum and linear equality constraints are considered. Given that the number of local cost functions can be large as well as the number of constraints, a stochastic optimization method is proposed. The method belongs to the class of variable sample size first order methods, where the sample size is adaptive and governed by the additional sampling technique earlier proposed in the unconstrained optimization framework. The resulting algorithm may be a mini-batch method, increasing sample size method, or even deterministic in a sense that it  eventually reaches the full sample size,  depending on the problem and similarity of the local cost functions. Regarding the constraints, the method uses controlled, but inexact projections on the feasible set, yielding possibly infeasible iterates.  Almost sure convergence is proved under some standard assumptions for the stochastic framework, without imposing the convexity. 
Numerical results on relevant machine learning experiments, i.e.,
real-world data sets for logistic regression problems,  show that the proposed algorithm is competitive with the state-of-the-art methods. 

\end{abstract}


\section{Introduction}

 We consider constrained optimization problems with the objective function in the form of weighted finite sum and linear equality constraints, i.e., 

 \begin{equation} \label{problem}
\min_{Ax=b} f(x):=\sum_{i=1}^{N} w_i f_i(x),
\end{equation}
where $ f_i:\mathbb{R}^n \to \mathbb{R}, i=1,...,N$ are continuously-differentiable functions, $w_1,...,w_N$ represent the weights such that  
\begin{equation} \label{w}
    \sum_{i=1}^{N} w_i=1, \quad w_i\geq 0, \; i=1,..,N,
\end{equation}
$b\in \mathbb{R}^m$ and $A \in \mathbb{R}^{m \times n}$ is assumed to be a full-rank matrix,  $rank(A)=m \leq n$. 

   The considered problems come from different fields, mainly including Big Data problems commonly present in Machine Learning (ML) \cite{EJORsurvey}. 
  Many of ML problems are in the form of finite sum optimization problems, so a large class of stochastic methods for constrained finite sum optimization problems are proposed in many papers (\cite{LevelSetMethods, FrankWolfe} to name just a few).
  Notice that the  choice of $w_i=1/N$ for all $i=1,...,N$ yields the standard form of the finite sum objective function $f(x)=\frac{1}{N} \sum_{i=1}^{N} f_i(x)$.
  Optimization problems with this objection function and linear constraints have gained significant attention in practical applications (see \cite{EJORcomposite}).   
  Introduction of possibly different weights is motivated by the so called local regression models (see \cite{UMU} for instance) in ML where the model parameters are recalculated for any new data point depending on its position in the space of attributes. In such approach, the objective function usually takes into account the distance between the new point and  data points from the training data set. The aforementioned distances represent the weights in problem \eqref{problem}.  On the other hand, the weights $w_i$ can be viewed as probabilities of choosing the corresponding functions $f_i$ and, in that case,  the objective function represents mathematical expectation. The motivation for  emphasizing weights and observing them separately from  the functions $f_i$ comes from stochastic framework. Namely, this allows  stochastic algorithms to favor the functions that are more important by giving them better chances to be chosen as explained in the sequel of the paper. Moreover, some of the ML problems also include linear constraints. One of the examples is a  Ridge regression   which can be stated in the form of constrained optimization problem \cite{UMU}. Another example would be  Markowitz utility function minimization, i.e., the problems of finding optimal portfolio  that minimizes the risk and maximizes the return, while the sum of unknowns must be equal to one.  The two mentioned examples include only modest number of constraints and the projection on the feasible set is not a big challenge. However,  data fitting problems in general (e.g. least squares) may also include a large number of (linear) constraints (see e.g. \cite{LSEC} for further references). In that case, projecting on the feasible set is too expensive and inexact projections may be a better option. 

A variety of line search and trust region methods have been proposed to solve nonlinear constrained problems. 
Various algorithms have been designed to solve deterministic equality-constrained optimization problems (see \cite{kapica, Brust}
for further references),
while recent research has focused on developing stochastic optimization algorithms. 
There has been a growing interest in adapting line search and trust region methods in stochastic framework for unconstrained optimization problems \cite{pregledni, SBNKNKJMR, BellaviaKrejicMorini,  BlanchetCartis, Bollapragada, ChenMenickelly, LSOS,  iusem, NKNKJ2, asntr, KLOS, 
lsnmbb}, but
significantly fewer algorithms have been proposed to solve stochastic equality-constrained
optimization problems (see \cite{kapica} for further references and \cite{BerahasXieZhou, CurtisRobinsonZhou, kolar, OztoprakByrdNocedal, SunNocedal, wang curtis}).

Projected gradient methods can be used to solve constrained optimization problems, (see \cite{Birgin1}, \cite{ROSEN}). Their  generalization to stochastic framework have been investigated in \cite{NKNKJ, loreto}.
A novel class of projected gradient methods for constrained minimization both in deterministic and stochastic settings is proposed in \cite{LanLiXu}. Large scale problems require inexact projections because of computational cost.

Numerous first order and second order algorithms have been developed for solving unconstrained finite sum minimization \cite{pregledni, SBNKNKJMR, BellaviaKrejicMorini, LSOS, asntr, lsnmbb}.
These methods are based on nonmonotone line search or trust region technique.
The usefulness of line search nonmonotonicity in deterministic case %originally from {\color{green} 
\cite{grippo,li}
was also demonstrated in a stochastic case. A class of algorithms which uses nonmonotone line search rule fitting a variable sample size scheme at each iteration was proposed  in \cite{NKNKJ2}.

In large scale problems, the computation of the objective function and its gradient (and additionally Hessian if needed) is expensive, so their approximations are generally used in order to reduce the computational cost.  Sampling-based approaches have long played an important role in stochastic optimization and stochastic programming \cite{SHP}. Subsampling is a natural way of computing these approximations, and adaptive subsampling appears to be particularly suited to large dimensional problem in the form of finite sums. 

Adaptive sample size strategies for finite sum problems are presented in 
\cite{pregledni, SBNKNKJMR, LSOS, asntr, lsnmbb}
and some other types of adaptive subsample approach can also be found in \cite{BellaviaKrejicMorini, BlanchetCartis, Bollapragada,ChenMenickelly,  NKNKJ}.

Aditional sampling (two step sampling in each iteration) presented in \cite{LSOS, asntr, lsnmbb}   
plays an important role in the sample size scheduling. 
%for controlling the sample size. 	%It is used to control the resulting approximation error introduced by the line search procedure. 
It is used as a control for accepting the step and increasing the sample size if necessary.
 The additional sampling can be arbitrarily cheap, i.e., even the sample size 1 is sufficient, and hence 
 %it does not increase the computational cost significantly.
it does not make the process more expensive. 
Other additional line search and trust region sample size strategies can be found in \cite{iusem,KLOS}.

The method we propose here belongs to the family of projected gradient methods, with the step size determined by a nonmonotone line search rule. The specific form of the objective function, in particular the case of large $ N $ and $ n $ motivates the adaptive subsample approach. Thus we work with approximate objective function (and the corresponding gradient), in other words with  random linear models. The sample size that defines the model in each iteration is computed according to the  estimated  progress towards minimizer in each iteration. The approximate gradient direction is then projected inexactly to the feasible set, yielding a new iteration that might be infeasible in a controlled way. Inexact projection in fact means that we will solve the corresponding system of linear equations only approximately, using any linear solver. The progress is measured by an additional approximation of the objective function, which can be very rough, in fact even the sample of size 1 will be suitable for this additional objective function approximation. Besides the measure of progress the additional sampling done in each iteration allows us to overcome theoretical difficulties that arise from the fact that the direction and step size are not independent random variables and hence one cannot apply the martingale theory for theoretical analysis. Such difficulties can be overcame with the  predetermined step size, for instance $ 1/k, $ which yields a.s. convergence in stochastic gradient descent framework (see e.g. \cite{pregledni} for further references). But in that case the step sizes become very small very fast and hence the method is very slow. Furthermore, the method we propose here is adaptive in the sense that sample size increase is problem dependent. Roughly speaking the similarity of functions $ f_i $ governs the process and the iterative procedure might end without reaching the full sample, i.e., with very cheap iterations, or the exact objective function and the gradient are used at the final stages of the iterative procedure. 

Random linear models that we use can be more or less similar to the original function $ f $ and thus the decrease of the model might not be a decrease for the true objective function. Therefore we use a nonmonomotone line search procedure \cite{grippo,li}, that is more relaxed in accepting the step than the classical Armijo rule. The same approach is exploited in several methods with random models, \cite{SBNKNKJMR, LSOS, NKNKJ, lsnmbb}. 
Putting together the random linear models, nonmonotone line search procedure, additional sampling and inexact projections we end up with an efficient and theoretically sound method that is almost surely (a.s.) converging to a stationary point of \eqref{problem}. 

 To summarize, our main contributions may be listed as follows:
\begin{itemize}
    \item The additional sampling concept for solving unconstrained finite sum problems \cite{LSOS, asntr, lsnmbb} is incorporated into the stochastic method to solve constrained optimization problems with weighted sum objective functions. 
    \item The proposed method relaxes the common assumption of feasible iterates in stochastic projected gradient framework (e.g. \cite{NKNKJ}) and allows controlled, but inexact projections which can be of great significance for problems with large number of constraints.   
    \item Almost sure convergence is proved   under rather standard assumptions for stochastic optimization framework. 
    \item The efficiency of the proposed method is confirmed through a number of tests performed on both academic and  real-world data problems. 
\end{itemize}

The paper is organized as follows.  Section 2 contains basic definitions and statements known from the literature.  It also provides some basic concepts and preliminary results needed for further analysis. The proposed method - IPAS is stated and explained in Section 3, while the convergence analysis is delegated to Section 4. Section 5 is devoted to numerical results, while the main conclusions are derived in the  final section.

\section{Preliminaries}

Let us start with explaining the inexact projections we will use in the algorithm. 
Under the assumption of fully ranked $A$, one can show that the orthogonal projection $\pi_{S}(y)$ of a point $y$ on the feasible  set $S:=\{x \in \mathbb{R}^n \; | \; Ax=b\}$ is given by 
\begin{equation}
    \label{tacnaprojekcija}
    \pi_{S}(y)=y-A^T(AA^T)^{-1} (Ay-b).
\end{equation}
The above equality comes from the fact that $\pi_{S}(y)=argmin_{Ax=b} \frac{1}{2} \|y-x\|^2$ and, since the orthogonal projection problem is convex, the solution is determined by the KKT conditions 
$$A \; \pi_{S}(y)=b, \quad A^T \lambda=y - \pi_{S}(y),$$
where $\lambda$ represents the vector of Lagrange multipliers. 
Multiplying the second equation with $A$ from the left and using the first one,  we obtain 
\begin{equation}
    \label{kkt2}
   AA^T \lambda= Ay-b.
\end{equation}
This together with  
$\pi_{S}(y)=y-A^T \lambda$ yields \eqref{tacnaprojekcija}. 
The important feature for our analysis lies in the fact that the projection operator to the set $ S $ has the following property. 
\begin{lemma}
    \label{lemaprojekcije}
    Let $ A \in \mathbf{R}^{m \times n }$ with $ rank(A) = m $ and $ w_i \geq 0, i=1,\ldots, N, \\ \sum_{i=1}^{N} w_i = 1. $ For any set of points $ y^i \in \mathbf{R}^n, i=1,\ldots,N $ there holds
 \begin{equation} \label{proconv} \pi_{S}(\sum_{i=1}^{N} w_i y^i) =  \sum_{i=1}^{N} w_i \pi_{S}(y^i).  \end{equation}
    \end{lemma}

{\em Proof. } 
Using \eqref{tacnaprojekcija}, there follows  \begin{eqnarray*}
    \pi_{S}(\sum_{i=1}^{N} w_i y^i)&=& \sum_{i=1}^{N} w_i y^i-A^T(AA^T)^{-1} (A\sum_{i=1}^{N} w_i y^i-b)\\
    &=& \sum_{i=1}^{N} w_i y^i-A^T(AA^T)^{-1} (A\sum_{i=1}^{N} w_i y^i-\sum_{i=1}^{N} w_ib)\\
    &=& \sum_{i=1}^{N} w_i (y^i-A^T(AA^T)^{-1} (A y^i-b))=\sum_{i=1}^{N} w_i \pi_{S}(y^i).
\end{eqnarray*} 
$ \Box$

To compute the exact projection on the set $ S $  one needs to solve the system \eqref{kkt2} exactly. In some applications, if  the number of equalities is modest one can use the closed form \eqref{tacnaprojekcija} for any given point $y$. However, if the dimension of the problem is very large and/or the number of equality constraints is large,  finding the exact solution to \eqref{kkt2} can be impractical. Thus,  we  will assume that the linear system is solved only approximately. The quality of inexact projection in each iteration will be controlled by the norm of the residual vector defined as follows. Let us denote by $\tilde{\pi}_S(y) $ the inexact projection of point $y$ on feasible set $S$, more precisely,
\begin{equation} \label{ptilda} 
     \tilde{\pi}_S(y)=y-A^T\tilde{\lambda}(y),
\end{equation}
where $\tilde{\lambda}(y)$ is an approximate solution of \eqref{kkt2}. The residual is denoted by  
\begin{equation} \label{ptildarez}  
r(y):=AA^T \tilde{\lambda}(y)-Ay+b,
\end{equation}
while the feasibility measure of  point $ y$ is defined as
\begin{equation}
    \label{emera}
    e(y) = \|Ay-b\|. 
\end{equation}
We will state the condition for the residual vector in each iteration a bit ahead. The following simple lemma will be used later on. 
\begin{lemma}
\label{projekcija}
    Let $ z \in \mathbb{R}^n$ be an arbitrary point and $\tilde{\lambda}(z) $ be an approximate solution of \eqref{kkt2} such that $\|r(z)\| \leq M $ with $ M>0.$ Then $e(\tilde{\pi}_{S}(z))\leq M. $
\end{lemma}
{\em Proof.} The condition 
$$ \|r(z)\| =\|A A^T \tilde{\lambda}(z) - Az + b \| \leq M$$ implies
\begin{eqnarray*}
e(\tilde{\pi}_{S}(z)) &=& \|A \tilde{\pi}_{S}(z) - b \|  =  \|A(z-A^T \tilde{\lambda}(z))-b\| \\
& = & \|Az-b-AA^T((AA^T)^{-1}(Az-b+r(z)))\| \\
& = & \|Az - b - Az + b - r(z)\| =\|r(z)\| \leq M,  
\end{eqnarray*}
which completes the proof. $\Box $

One can show that the projected gradient  direction of the form 
\begin{eqnarray}
    \label{pgd}
    d(x)=\pi_{S}(x-\nabla f(x))-x
\end{eqnarray}
is a descent direction for function $f$ at point $x \in S$ unless $x$ is a stationary point for problem \eqref{problem}. More precisely, the following result is known. 
\begin{theorem} \cite{Birgin1}
    \label{bmr} Assume that $f \in C^1(S)$ and $x \in S$. Then the projected gradient direction \eqref{pgd} satisfies: 
    \begin{itemize}
        \item[a)] $d(x)^T \nabla f(x) \leq -\|d(x)\|^2,$
        \item[b)] $d(x)=0$ if and only if $x$ is a stationary point for problem \eqref{problem}.
    \end{itemize}
\end{theorem}

The method we propose will be based on approximate objective function and the gradient. Let us denote by $f_{\N_k}$ the approximation of function $f$ used in iteration $k$, i.e., 
\begin{equation} \label{fnk}
f_{\N_k}(x):=\frac{1}{N_k}\sum_{i\in \N_k}  f_i(x),
\end{equation}
where $N_k:=| \N_k|$, $\N_k=\{i^k_1,...,i^k_{N_k}\}$, and  each $i^k_j \in \N_k$ takes the value $s \in \N:=\{1,...,N\}$ with probability $w_s$, i.e., 
\begin{equation} \label{Nk} P(i^k_j=s)=w_s, s=1,...,N, j \in \N_k, k \in \mathbb{N}.
\end{equation}
 This way we have an unbiased estimate of $f$, i.e., 
$$E(f_{\N_k}(x)|x)=\frac{1}{N_k}\sum_{j=1}^{N_k}  E(f_{i^k_j}(x)|x)=\frac{1}{N_k}\sum_{j=1}^{N_k}  f(x)=f(x).$$

As already stated,  in the proposed algorithm we use inexact projections of the approximate gradient $ \nabla f_{\N_k} $ 
%to obtain the search  direction $ p_k) $ 
which can yield nondescent directions $p_k$ and infeasible points $x_k$. Moreover, we deal with approximate functions and thus imposing monotone line search is not beneficial in general. Thus, we use nonmonotone Armijo-type line search \cite{li}, to determine the step size $t_k$ 
\begin{equation} \label{nonm} f_{\N_k}(x_k+t_k p_k)\leq f_{\N_k}(x_k)  +
c_1 t_k (\nabla f_{\N_k}(x_k))^T p_k+\varepsilon_k,
\end{equation}
with some $\varepsilon_k$ which satisfies 
\begin{equation} \label{sumaeps}   \sum_{k=0}^{\infty} \varepsilon_k \leq \bar{\varepsilon} < \infty, \;\;\; \varepsilon_k>0.
\end{equation}

Clearly, the direction $p_k $ and the step size $ t_k $ obtained in \eqref{nonm} both depend on $ \N_k. $ Therefore we can not rely on the    martingale  theory commonly used  in stochastic gradient  methods such as SGD \cite{SGD} where the  predetermined step size sequence  is used.  To overcome this difficulty and avoid predefined step sizes we employ the idea of additional sampling.  This simple and computationally cheap remedy is successfully used in \cite{LSOS, lsnmbb}. Other possibilities are given in \cite{iusem, KLOS}.
The method proposed in \cite{asntr} also uses additional sampling but in trust-region framework.

Similarly as for $f_{\N_k}$, we form an additional sampling approximation $f_{\D_k}$  by
\begin{equation} \label{fdk}
f_{\D_k}(x):=\frac{1}{D_k}\sum_{i\in \D_k}  f_i(x),
\end{equation}
where $D_k:=| \D_k|$, $\D_k=\{l^k_1,...,l^k_{D_k}\}$, and  each $l^k_j \in \D_k$ takes the value $s \in \N:=\{1,...,N\}$ with probability $w_s$, i.e., 
\begin{equation} \label{Dk} P(l^k_j=s)=w_s, s=1,...,N, j \in \D_k, k \in \mathbb{N}.
\end{equation}
The key point for the efficiency of this approach lies in the fact that the cardinality of $ \D_k$ is arbitrary, with only requirement being $ D_k\leq N-1.$ Thus one can even take $ D_k=1$ in each iteration. The numerical experiments presented in Section 4 are performed with $ D_k=1. $

The following technical lemmas will be used further on.

\begin{lemma} \label{konvergencija}
 Let $ e_{k+1} \leq \theta e_k + \eta_k$ for all $k \geq 0$ with $  e_k\geq 0, k=1,2,..., \theta \in [0,1)$ and $ \{\eta_k\}$   satisfying  $ \lim_{k \to \infty}  \eta_k=0. $ Then $$ \lim_{k\to \infty} e_k = 0. $$ 
\end{lemma}
{\em Proof. } 
Applying the induction argument we can show that 
$$e_k\leq \theta^k e_0+s_k,$$
where $s_k=\sum_{j=1}^{k} \theta^{j-1} \eta_{k-j}$. Under the stated conditions there holds 
$\lim_{k \to \infty} s_k=0$ (see \cite[Lemma 3.1, (a)]{nedic}) and we conclude that $\lim_{k\to \infty} e_k=0$.

$ \Box $

\section{The Method}

 The method we consider will generate an infinite sequence of iterations $ x_k.$ As already stated in each iteration we will use a subsample $ \N_k$ and the corresponding function $ f_{\N_k}(x_k)$ and the gradient $ \nabla f_{\N_k}(x_k). $ The inexact projection $ \tilde{\pi}_S $ will be used to generate the search direction with the approximation error controlled by a nonincreasing sequence of  positive numbers $ \eta_k $ such that 
 \begin{equation}
    \label{etak}
    \sum_{k=0}^{\infty} \eta^2_k\leq \bar{\eta}<\infty.
\end{equation}
Thus, for $y_k=x_k-\nabla f_{\N_k}(x_k)$ we will require that the projection residuals satisfy 
\begin{equation}
    \label{approxlambda0}
   \|r(y_k)\|\leq \eta_k. 
\end{equation}
Clearly, the inexactness of the projection will decrease as $ k $ increases and eventually we will approach the feasible set.   The search direction $ p_k $ will be computed as usual, as the difference of inexact projection $\tilde{\pi}_{S}(y_k)$ and the current approximation $ x_k. $ After the computation of $ p_k$ we distinguish two cases. If $ N_k<N $ we proceed to determine the step size by the nonmonotone rule \eqref{nonm} and define $ \bar{x}_k = x_k + t_k p_k$. 
%The sequence of parameters $ \varepsilon_k$ is taken as $ \varepsilon_k = \eta_k^2$ in Step 3 of the algorithm, \eqref{ls}. Theoretically, we can take another sequence in the line search but this way the inexactness of the projection  is taken into account in computation of the step size $ t_k$ and we allow more freedom for lower accuracy projection and increase the requirements as $ k $ increases.  
Adding $ \varepsilon_k $ in the decrease condition allow us to compute a suitable step size even if the direction  is nondecreasing and hence the line search rule \eqref{ls} is always well defined,  even without the lower bound $t_{min}$ posed in the mini-batch case of Step S3.  

In the case $ N_k=N, $ i.e., if the full precision in the objective function is needed,  the projection could still be inaccurate and hence we might search along infeasible direction. In that case we check if the search direction is decreasing. If yes, we proceed to the line search. Else, if the direction is not sufficiently decreasing, we discard the search direction and take a new projection of the current iteration $ x_k$ to get $ x_{k+1} $ and terminate the iteration. In this case we will call the iteration unsuccessful.  

For all iterations where $ N_k < N $ in Step 4 of the algorithm we perform additional sampling, taking a new sample $ \D_k,$ independently of $ \N_k.  $
As already mentioned, this step is meant to be a computationally cheap measure of progress as $ \D_k$ can be arbitrary small and we will work with $ D_k = 1$ in our numerical tests. For this new sample $ \D_k$ we compute the direction $ u_k = x_k - \nabla f_{\D_k}(x_k) $ and then project it approximately to get $ s_k = \tilde{\pi}_S(u_k)-x_k=-\nabla f_{\D_k}(x_k)-A^T \tilde{\lambda}_k(u_k).$ Here we keep the same accuracy of the projection as in Step 2 for $ p_k. $ Notice that for $ N_k = N$ this additional sampling is not needed as the line search in Step 3 is performed with the true objective function. 

Finally, at Step 5 we update the iteration.  
So, if we have enough decrease in the objective function $ f_{\D_k}, $ according to \eqref{ac}, we update the iteration and keep the same sample size for the next iteration. Roughly speaking we are saying here that $f_{\N_k}$ is a good approximation of the objective function as the decrease condition holds for another (independently sampled) function $ f_{\D_k}. $ Notice that the condition \eqref{ac}
 is looser than the step size rule \eqref{ls} as  $\varepsilon_k$ is multiplied with some constant $ C $ which can be large. 
 
\noindent {\bf Algorithm 1: IPAS} (Inexact Projection with Additional Sampling)
\label{SPGNS}
\begin{itemize}
\item[\textbf{S0}] \textit{Initialization.} Input: $x_0 \in \mathbb{R}^n, N_0 \in \mathbb{N}, \beta, c,c_1, t_{min} \in (0,1), C>0$, $\{\eta_k\}$ satisfying \eqref{etak}, $\{\varepsilon_k\}$ satisfying \eqref{sumaeps}, $k:=0$.
\item[\textbf{S1}] \textit{Subsampling.} If $N_k<N$, choose $\N_k$ via \eqref{Nk}. Else, set $f_{\N_k}=f$.
\item[\textbf{S2}] \textit{Search direction.} Compute  
\begin{eqnarray}
    \label{dnk}
    p_k = \tilde{\pi}_S(y_k)-x_k=-\nabla f_{\N_k}(x_k)-A^T \tilde{\lambda}_k(y_k)
\end{eqnarray}
with $y_k=x_k-\nabla f_{\N_k}(x_k)$,  and $\tilde{\lambda}_k(y_k)$ satisfying 
\eqref{approxlambda0}.\\ If $N_k<N$  go to step S3.\\
If $N_k=N$,  
and  
\begin{eqnarray} \label{dnkN}
    (\nabla f(x_k))^T p_k \leq -c \|p_k\|^2
\end{eqnarray}
go to step S3. \\ Else set 
%$p_k=\tilde{\pi}_S(x_k)-x_k$ and go to step S3.
$x_{k+1}=\tilde{\pi}_S(x_k)$ with $\tilde{\lambda}_k(x_k)$ satisfying $\|r(x_k)\| \leq \eta_k,$ set $k=k+1$ and go to step S1.
\item[\textbf{S3}] \textit{Step size.}  If $N_k=N$, find  the smallest $j \in \mathbb{N}_0$ such that $t_k=\beta^j$ satisfies 
\begin{eqnarray}
    \label{ls}
    f_{\N_k}(x_k+t_k p_k)\leq f_{\N_k}(x_k)  + c_1 t_k (\nabla f_{\N_k}(x_k))^T p_k+\varepsilon_k.
\end{eqnarray}
 Else, if $N_k<N$, starting with $t_k=1$, while $t_k \geq t_{min}$  and $$
    f_{\N_k}(x_k+t_k p_k)> f_{\N_k}(x_k)  + c_1 t_k (\nabla f_{\N_k}(x_k))^T p_k+\varepsilon_k,$$
 reduce $t_k$ by factor $\beta$.

Set $\bar{x}_k=x_k+t_k p_k$. 
\item[\textbf{S4}] \textit{Additional sampling.} \\ If $N_k=N$, set $x_{k+1}=\bar{x}_k$, $k=k+1$ and go to step S1. \\Else choose $\D_k$ via \eqref{Dk} and compute \begin{eqnarray}
    \label{ddk}
    s_k = \tilde{\pi}_S(u_k)-x_k=-\nabla f_{\D_k}(x_k)-A^T \tilde{\lambda}_k(u_k)
\end{eqnarray}
with $u_k=x_k-\nabla f_{\D_k}(x_k)$ and $\tilde{\lambda}_k(u_k)$ satisfying $\|r(u_k)\| \leq \eta_k.  $
%$$\|AA^T \tilde{\lambda}_k(u_k)- Au_k+b\|\leq \eta_k$$
%\eqref{approxlambda0}.  
\item[\textbf{S5}] \textit{The update.} If \begin{eqnarray}
    \label{ac}
    f_{\D_k}(\bar{x}_k)\leq f_{\D_k}(x_k)  - c  \|   s_k\|^2+C\varepsilon_k,
\end{eqnarray}
%and \begin{eqnarray}
%    \label{dopustivost}
%    \|A \bar{x}_k-b\|\leq \rho  \|A x_k-b\|+C \eta_k,
%\end{eqnarray}
set $x_{k+1}=\bar{x}_k$ and $N_{k+1}=N_k$. \\Else set $x_{k+1}=x_k$, choose $N_{k+1} \in \{N_{k}+1,...,N\}$. \\Set $k=k+1$ and go to step S1.  
\end{itemize}

If the condition \eqref{ac} is not satisfied we take $ x_{k+1} = x_k$ and increase the sample size $ N_{k+1}$ for the new iteration. Essentially we reason here that the trial iteration should not be accepted because we need more precision and hence we increase the sample size.  

\noindent {\bf Remark.} An important point here is to notice that we do not need to specify how to increase the sample size if needed, i.e., we are free to choose any $ N_{k+1} > N_k. $ Naturally, one can choose to increase the sample size very slowly to keep the iterations computationally cheap, or to increase the sample size faster to achieve a better approximation of the objective function and the gradient, hoping to end up the process in fewer iterations. Clearly, the sample size scheduling is problem dependent.

\section{Convergence analysis}

\begin{assumption}
\label{A1} Each function $f_i, i=1,...,N$ is continuously differentiable with Lispchitz continuous gradient and bounded from below by a constant $f_{low}$. 
\end{assumption}
This assumption implies that  $f(x) \geq f_{low}$ for all $x \in \mathbb{R}^n$ and it also holds for any approximate function $f_{\N_k}$ and $f_{\D_k}$.  Furthermore all approximate gradients are also Lipschitz continuous and without loss of generality we may assume that $ L > 0 $ is a common Lispchitz constant.

The algorithm we consider generates a set of random iterations $ \{x_k\}.  $ Nevertheless some properties hold for all iterations, independently of the sample we use to generate them. First of all, we can show using the standard arguments and Assumption A\ref{A1} that the step size $ t_k $ generated in  Step 3, \eqref{ls} is bounded from below. 

\begin{lemma}
    \label{Lemma1} Assume that Assumption A\ref{A1} holds and that step size $ t_k $ is computed in Step 3 of IPAS algorithm. Then  
     $ t_k \geq t_{\min} $  provided that  
    $t_{\min}< \min\{ 1,\frac{2 \beta c (1-c_1)}{L}\}. $
    %\frac{2 \beta c (1-c_1)}{L} ,  $$
    %for $ c $  sufficiently small. 
\end{lemma}
{\em Proof.}  If the full sample is reached, the line search is performed only if \eqref{dnkN} holds and it can be proved (see \cite{lsnmbb} for example) that $t_k \geq \frac{2 \beta c (1-c_1)}{L}$. On the other hand, if $N_k<N$, the line search yields  $t_k\geq t_{min}$ by the construction of Step 3 of IPAS algorithm. 
$\Box$

Furthermore, we can prove that the feasibility of iterations $ \bar{x}_k $ is eventually increasing although the projections are inexact. The following statement holds. 

\begin{lemma}
    \label{lemafes}
  Assume that Assumption A\ref{A1} holds. Then 
  \begin{equation}
      \label{dopustivost}
  e(\bar{x}_k) \leq (1-t_{\min}) e(x_k) +  \eta_k.
   \end{equation}
\end{lemma}
{\em Proof.} Given that 
\begin{eqnarray*}  \bar{x}_k & = & x_k + t_k p_k = x_k + t_k (\tilde{\pi}_S(y_k)-x_k) \\
    & = & (1-t_k)x_k + t_k \tilde{\pi}_S(y_k), 
\end{eqnarray*}
and 
\begin{eqnarray}
    e(\bar{x}_{k}) &=& \|A((1-t_k)x_k + t_k\tilde{\pi}_S(y_k)) - b \|\\
    & \leq & (1-t_k) \|Ax_k - b\|+ t_k\|A\tilde{\pi}_S(y_k) - b\| \\
    &= & (1-t_k) e(x_k) + t_k e(\tilde{\pi}_S(y_k)). 
\end{eqnarray}
Since $ t_k \geq t_{\min} $  and the residual of $ \tilde{\pi}_S(y_k) $ satisfies \eqref{approxlambda0}, by Lemma \ref{projekcija} we have $e(\tilde{\pi}_S(y_k)) \leq \eta_k. $ Therefore, the above inequalities imply
$$ e(\bar{x}_k) \leq (1-t_{\min})e(x_k) + \eta_k.  $$
$\Box$

Let us denote  by $\D_k^+$ the subset of all possible outcomes of $\D_k$ at iteration $k$ for which the condition \eqref{ac} is satisfied, i.e.,
\begin{equation} \label{dkp}
    \D_k^+= \{\D_k \subset \N \; \vert \; f_{\D_k}(\bar{x}_k)\leq f_{\D_k}(x_k)  - c  \|   s_k\|^2+C\varepsilon_k \}.
\end{equation}
We denote the complementary subset of outcomes at iteration $k$ by
\begin{equation} \label{dkm}
    \D_k^{-}= \{\D_k \subset \N\; \vert \; f_{\D_k}(\bar{x}_k)> f_{\D_k}(x_k)  - c  \|   s_k\|^2+C\varepsilon_k \}.
\end{equation}
%{ }
Although the problem we consider is constrained and the algorithm is quite different from the one in \cite{asntr, lsnmbb} with sampling that is not uniform, the following lemma, similar to the \cite[Lemma 1]{lsnmbb} holds. Essentially,  it says that either $ N_k =N $ for $ k $ large enough or the condition \eqref{ac} is satisfied infinitely many times. We state the proof for completeness.
\begin{lemma}\label{Lemma-2cases}
    Suppose that Assumption A\ref{A1} holds. If  $N_k<N$ for all $k \in \mathbb{N}$, then a.s. there exists   $k_1 \in \mathbb{N}$ such that $\D_k^{-}=\emptyset$ for all $k \geq k_1$.
\end{lemma}
\begin{proof}  Assume that $N_k<N$ for all $k \in \mathbb{N}$. Since the sample size sequence $\{N_k\}$ in  IPAS Algorithm  is non-decreasing there exists some $\overline{N}<N$   such that $N_k=\overline{N}$ for all $k $ large enough. Now, let us assume that there is no $k_1 \in \mathbb{N}$ such that $\D_k^{-}=\emptyset$ for all $k \geq k_1$. Then  there exists an infinite sub-sequence of iterations $K \subseteq \mathbb{N}$ such that  $\D_k^{-}\neq \emptyset$ for all $k \in K$. Since $\D_k$ is chosen with finitely many possible outcomes with the same distribution  for each $k$, there exists  $q>0$ such that  $\mathbb{P}(\D_k \in \D_k^-)\geq q$ for all $k \in K. $
In fact, given \eqref{Dk} and the fact that $D_k \leq N-1$ for each $k$, we can conclude that $q=(\min_{s \in \{1,2,...,N\}}\{w_s\})^{N-1}$. So, we have
$$\mathbb{P}(\D_k \in \D_k^+, k \in K)\leq \Pi_{k \in K} (1-q)=0.$$
Therefore we will almost surely encounter an iteration at which the sample size  will be increased due to violation of the condition \eqref{ac}. 
This is a contradiction with the condition  $N_k = \overline{N}$ for all $ k $ large enough and we conclude that the statement holds.
\end{proof}

As already stated, depending on the problem, IPAS algorithm yields two possibilities - either we generate an infinite sequence $ \{x_k\} $ such that $ N_k < N $ for all $ k, $ or we end up with $ N_k = N $ for $ k  $ large enough.  So, the convergence analysis will cover these two cases separately. Let us first consider the mini-batch case, i.e. $ N_k < N $ for all $ k \in \mathbb{N}. $ 

The following lemma quantifies the progress in term of the (true) objective function and exact projections in the mini-batch case.

\begin{lemma}
    \label{Lemma-NEWaccept}
    Suppose that Assumption A\ref{A1} holds.  If  $N_k<N$ for all $k \in \mathbb{N}$, then a.s. 
    $$f(x_{k+1}) \leq f(x_{k})- \frac{c}{2} \| d(x_k) \|^2+C_L \max \{\varepsilon_k, \eta^2_k\}   $$ 
%   {\color{blue} and 
%    $$\|A x_{k+1}-b\|\leq \rho  \|A x_k-b\|+C \eta_k$$
     holds for all  $k\geq k_1 $ and some constant $C_L$, where $k_1$ is as in Lemma \ref{Lemma-2cases} and $d(x_k)$ is given by (\ref{pgd}).
\end{lemma}

\begin{proof}  First, we prove that
 $$f(\bar{x}_k) \leq f(x_{k})- \frac{c}{2} \| d(x_k) \|^2+C_L\max \{\varepsilon_k, \eta^2_k \}$$ holds a.s. for all $k\geq k_1 $, 
      where $k_1$ is as in Lemma \ref{Lemma-2cases} and some constant $C_L$.

Notice that Lemma \ref{Lemma-2cases} implies that  a.s. \eqref{ac} holds for all possible realizations of $\D_k$ and for all $k\geq k_1$. Thus, we conclude that  a.s.  for every $i=1,2,...,N$ and every $k\geq k_1$ we have 
\begin{equation} 
\label{wccp1} 
f_{i}(\bar{x}_k)\leq f_{i}(x_k)  - c  \|   z_{k}^i\|^2+C\varepsilon_k, \end{equation}
where $z_k^i$ denotes the direction obtained for $\D_k=\{i\}$ in Step 4 i.e., 
\begin{eqnarray}
    \label{ddki}
    z_k^{i} = u^i_k-A^T \tilde{\lambda}_k(u^i_k)-x_k
\end{eqnarray}
where $u^i_k=x_k-\nabla f_{i}(x_k)$ and $\tilde{\lambda}_k(u^i_k)$ satisfies \begin{equation}
    \label{approxlambda}
    \|AA^T \tilde{\lambda}_k(u^i_k)- Au^i_k+b\|\leq \eta_k.
\end{equation}  
Indeed, if there exists $i \in \N$ that violates \eqref{ddki}, then there would exists at least one possible realization of $\D_k$ (namely, $\D_k=\{i,i,...,i\}$) that violates \eqref{wccp1} and thus we would have $\D_k^{-}\neq \emptyset$.
Let us denote the residual by $r_k^i$, i.e., we have 
\begin{equation}
     \label{res}
    r_k^i= AA^T \tilde{\lambda}_k(u^i_k)- Au^i_k+b, \quad \|r_k^i\|\leq \eta_k,
 \end{equation}
 and 
 \begin{equation}
     \label{res2}
    \tilde{\lambda}_k(u^i_k)=(AA^T)^{-1}r_k^i+  (AA^T)^{-1}(Au^i_k-b). 
 \end{equation}
 Moreover, 
 \begin{equation}
     \label{res3}
    z_k^{i} = u^i_k-A^T (AA^T)^{-1}r_k^i-A^T(AA^T)^{-1}(Au^i_k-b)-x_k. 
 \end{equation}
Next, multiplying both sides of \eqref{wccp1} with $w_i$ satisfying \eqref{w} and summing over all $i \in \N$ we obtain that a.s. the following holds for all $k\geq k_1$
\begin{equation}
    \label{p1}
    f(\bar{x}_k)\leq f(x_k)  - c \sum_{i=1}^{N} w_i \|   z_k^{i}\|^2+C\varepsilon_k.
\end{equation}
Further, writing $x_k=\sum_{i=1}^{N} w_i x_k$ and using \eqref{proconv} in Lemma \ref{lemaprojekcije} with $y^i=x_k-\nabla f_i(x_k)$
 we obtain that the exact projection direction related to the original objective function can be represented as follows
 \begin{eqnarray}
     \label{pd}
     d(x_k)&=& \pi_S (x_k-\nabla f(x_k))-x_k\\\nonumber 
     &=& \pi_S (\sum_{i=1}^{N} w_i(x_k-\nabla f_i(x_k)))-\sum_{i=1}^{N} w_i x_k\\\nonumber 
     &=&\sum_{i=1}^{N} w_i (\pi_S (x_k-\nabla f_i(x_k))-x_k)\\\nonumber 
     &=&\sum_{i=1}^{N} w_i (z_k^i+\pi_S (x_k-\nabla f_i(x_k))-x_k-z_k^i)\\\nonumber 
     &=&\sum_{i=1}^{N} w_i (z_k^i+\pi_S (u_k^i)-x_k-z_k^i).
 \end{eqnarray}
Further, using \eqref{w} and the convexity of $\|\cdot \|^2$, we obtain 
 \begin{eqnarray}
     \label{pd2}
     \|d(x_k)\|^2 &\leq &  \sum_{i=1}^{N} w_i \|z_k^i+\pi_S (u_k^i)-x_k-z_k^i\|^2\\\nonumber 
     &\leq&  2 \sum_{i=1}^{N} w_i \|z_k^i\|^2+2\sum_{i=1}^{N} w_i\|\pi_S (u_k^i)-x_k-z_k^i\|^2.
 \end{eqnarray}
 Now, let us estimate $\|\pi_S (u_k^i)-x_k-z_k^i\|^2$. According to \eqref{tacnaprojekcija} and \eqref{res3} we obtain
 \begin{eqnarray}
     \label{pd3}
     \pi_S (u_k^i)-x_k-z_k^i &= &  u_k^i-A^T(AA^T)^{-1} (Au_k^i-b)-x_k \\\nonumber
     &-& (u^i_k-A^T (AA^T)^{-1}r_k^i-A^T(AA^T)^{-1}(Au^i_k-b)-x_k)\\\nonumber
     &=&A^T (AA^T)^{-1}r_k^i.
 \end{eqnarray}
 Thus, for $C_{A}= \|A^T (AA^T)^{-1}\|$ we get 
 \begin{equation}
     \label{pd4}
     \|\pi_S (u_k^i)-x_k-z_k^i\|^2\leq \|A^T (AA^T)^{-1}\|^2 \|r_k^i \|^2\leq C_{A}^2 \eta_k^2.
 \end{equation}
 Therefore, from \eqref{pd2} we obtain 
 \begin{eqnarray}
     \label{pd5}
    -\sum_{i=1}^{N} w_i \|z_k^i\|^2   &\leq &  -\frac{1}{2} \|d(x_k)\|^2+\sum_{i=1}^{N} w_i\|\pi_S (u_k^i)-x_k-z_k^i\|^2\\\nonumber
    &\leq& -\frac{1}{2} \|d(x_k)\|^2+C_{A}^2 \eta_k^2.
 \end{eqnarray}
 Now, \eqref{p1} implies 
 \begin{equation}
    \label{pd6}
    f(\bar{x}_k)\leq f(x_k)  - \frac{c}{2} \|d(x_k)\|^2  +c C_{A}^2 \eta^2_k+C\varepsilon_k=:f(x_k)  - \frac{c}{2} \|d(x_k)\|^2  +C_L \max \{\varepsilon_k, \eta^2_k\}.
\end{equation}
Next, we show that in the mini-batch scenario we accept  the trial point $\bar{x}_k$ for all $k\geq k_1$. 
Since we know that 
 %, Step S5 implies the existence of $k_2$ such that \eqref{dopustivost} holds for all $k \geq k_2$ - otherwise we would reach the full sample eventually. Moreover, 
  a.s. $\D_k^{-}=\emptyset$ for all $k\geq k_1$, i.e., \eqref{ac} is satisfied, therefore, a.s., for all $k \geq k_1 $  the trial point is accepted, i.e., $x_{k+1}=\bar{x}_k$, so the statement holds due to (\ref{pd6}).
\end{proof}

Now, let us denote by $\Omega$ all possible sample paths of the IPAS algorithm. Further, let us denote by $\A \subseteq \Omega$ all possible sample paths that yield mini-batch scenario considered in the previous proposition and by $\bar{\A} \subseteq \Omega$ all possible sample paths that reach the full sample size eventually, i.e., the complement of $\A$. We use the following notation for the corresponding conditional expectations 
$$\mathbb{E}_{\A} ( \cdot):=\mathbb{E}( \cdot \vert \A), \quad \mathbb{E}_{\bar{\A}} ( \cdot):=\mathbb{E}( \cdot \vert \bar{\A}).$$

In order to prove  the convergence results, we impose the following assumption similar to one in \cite{lsnmbb}.

\begin{assumption}
    \label{assNEW}
There exists a constant $C_{\A}$ such that $\mathbb{E}_{\A} ( \vert f(x_{k_1})\vert ) \leq C_{\A}$, where $k_1$ is as in Lemma \ref{Lemma-2cases}.
\end{assumption}

 The following result shows that $d(x_k)$ defined in (\ref{pgd}) converges to zero a.s. in the mini-batch scenario. 

\begin{corollary} \label{das-mini}
    Suppose that the assumptions of Lemma \ref{Lemma-NEWaccept} hold together with Assumption A\ref{assNEW}. Then $$P(\lim_{k \to \infty} d(x_k) =0|\A)=1$$
    and every accumulation point of the sequence $\{x_k\}$ is a stationary point of problem \eqref{problem} a. s.
\end{corollary}
\begin{proof}
    Lemma \ref{Lemma-NEWaccept} implies that  a.s. 
    $$f(x_{k+1}) \leq f(x_{k})- \frac{c}{2} \| d(x_k)\|^2+C_L \max \{\varepsilon_k, \eta^2_k\}$$  holds for all    $k\geq k_1 $. Applying the conditional expectation $\mathbb{E}_{\A}$, by the induction argument we obtain that for each $j$ there holds
$$\mathbb{E}_{\A}(f(x_{k_1+j})) \leq \mathbb{E}_{\A}(f(x_{k_1}))- \frac{c}{2} \sum_{i=0}^{j-1}\mathbb{E}_{\A}(\| d(x_{k_1+i})\|^2)+C_L\sum_{i=0}^{j-1}\max \{\varepsilon_{k_1+i}, \eta^2_{k_1+i} \}.$$
Now, by using Assumptions A\ref{A1}, A\ref{assNEW}, and \eqref{etak}, letting $j $ tend to infinity we get  
$$f_{low} \leq C_{\A} - \frac{c}{2} \sum_{i=0}^{\infty}\mathbb{E}_{\A}(\| d(x_{k_1+i})\|^2)+C_L \max\{ \bar{\varepsilon},\bar{\eta} \} $$
and we conclude that 
$$\sum_{i=0}^{\infty}\mathbb{E}_{\A}(\| d(x_{k_1+i})\|^2) <\infty.$$
Now, by the extended version of Markov's inequality we have that for any $\epsilon>0$
\begin{equation*}
    \mathbb{P}(\|d(x_{k_1+i})\|\geq \epsilon\vert \A)\leq \frac{\mathbb{E}_{\A}(\|d(x_{k_1+i})\|^2)}{\epsilon^2}
\end{equation*}
which implies
\begin{equation*}
    \sum_{i=0}^{\infty} \mathbb{P}(\|d(x_{k_1+i})\|\geq \epsilon\vert \A)<\infty
\end{equation*}
and Borel-Cantelli Lemma \cite{klenke} implies that $P(\lim_{i\to\infty} \|d(x_{k_1+i})\|=0\vert \A)=1$, i.e., 
\begin{equation} \label{d0A} P(\lim_{i\to\infty} d(x_{k_1+i})=0\vert \A)=1.
\end{equation}

    Now, recall that  $e_k$ is the measure of infeasibility of $x_k$ defined in \eqref{emera}, i.e., $e_k:=  \|A x_k-b\|$.  Lemma \ref{lemafes} implies that $$e_{k+1} \leq (1-t_{\min}) e_k + \eta_k$$ for all $k \geq k_1$. Therefore, due to Lemma \ref{konvergencija} we obtain $\lim_{k \to \infty} e_k=0$. 
    Let us denote by $\tilde{x}=\lim_{k \in K} x_k$ an arbitrary accumulation point of IPAS algorithm in the mini-batch scenario. Then, we have that $\|A \tilde{x}-b\|=\lim_{k \in K} e_k=0$ and we conclude that $\tilde{x}$ is feasible.  Moreover, due to \eqref{d0A} and continuity of $d$, a.s. $$d(\tilde{x})=\lim_{k \in K} d(x_k) =0$$   and we conclude that $\tilde{x}$ is  a.s. a stationary point of \eqref{problem} according to Theorem \ref{bmr}. $ \Box $

\end{proof}

Now, we analyze the case where the full sample is reached for some $  k_2 \in \mathbb{N}_0$ and therefore $ N_k = N$ for $ k\geq k_2. $ We will distinguish between two types of iterations for $ k \geq k_2, $ successful and unsuccessful, represented by sets of indices $K_{su} $ and $ K_{un}. $ An iteration $ x_{k}$,  $ k \geq k_2$ is unsuccessful if the condition \eqref{dnkN} does not hold, i.e., if the direction $ p_k $ is not sufficiently decreasing  and hence $ x_{k+1} = \tilde{\pi}_S(x_k) $ and $ r(x_k) \leq \eta_k. $ Otherwise, $ x_{k}$ is successful and the new iteration is determined by the line search rule \eqref{ls} and $ x_{k+1} = \bar{x}_k. $ 

The following Lemma states that the sequence 
$\{ x_{k} \}_{k\geq k_2}$ converges to a feasible point and hence each accumulation point is feasible. We will rely on Lemmas \ref{projekcija} and \ref{konvergencija} from Section 2.

\begin{lemma}
    \label{lema2}
    Assume that Assumption A\ref{A1} holds and let $ \{x_k\}$ be an iterative sequence generated by IPAS algorithm such that $ N_k = N$ for $ k \geq k_2 $. Then $ \lim_{k \to \infty} e(x_k) = 0 $ and each accumulation point is feasible. 
\end{lemma}
{\em Proof. }  We will distinguish three different  cases: 1) all iterations are successful for $ k $ large enough; 2) all iterations are unsuccessful for $ k $ large enough; and 3) we have an infinite sequence of  successful and an infinite sequence of unsuccessful iterations.

Case 1. Assume, without loss of generality,  that for all $ k \geq k_2 $ we have $ k \in K_{su}. $
In that case we have $ x_{k+1} = \bar{x}_k $ and by Lemma \ref{lemafes} 
$$ e(x_{k+1}) \leq (1- t_{\min}) e(x_k) + \eta_k, $$ so the statement holds by Lemma \ref{konvergencija}. 

Case 2. Without loss of generality assume that $ k \in K_{un}$ for all $ k \geq k_2$.  Then we have $$ x_{k+1} = \tilde{\pi}_S(x_k)$$ for all $ k $ and by Lemma \ref{projekcija}
$$ e(x_{k+1}) \leq \eta_k. $$   
Given that $ \lim_{k\to \infty} \eta_k = 0 $ we obtain $ \lim_{k \to \infty} e(x_k) = 0. $

Case 3. Assume now that  both $ K_{su}$ and $ K_{un}$ are infinite. Let $ k-1 \geq k_2$ and assume that $ k-1 \in K_{un}, \; k, \ldots,k+j-1 \in K_{su} $ and $ k+j \in K_{un}. $ Then we have 
$$ e(x_{k}) = e(\tilde{\pi}_S(x_{k-1})) \leq \eta_{k-1} $$ and for each $i=1,...,j$, by Lemma \ref{lemafes} there holds
$$ e(x_{k+i}) \leq \theta e(x_{k+i-1}) + \eta_{k+i-1},$$ with $\theta:=1-t_{min} \in (0,1)$ and thus by the induction argument we get 
$$ e(x_{k+i}) \leq \theta^i \eta_{k-1}+...+\theta \eta_{k+i-2}+\eta_{k+i-1}.$$
Since $\{\eta_k\}$ is nonincreasing, for each $i=1,...,j$ there holds 
$$e(x_{k+i}) \leq \eta_{k-1} \sum_{t=0}^{i} \theta^t\leq \eta_{k-1} \sum_{t=0}^{\infty} \theta^t= \eta_{k-1}  \frac{1}{1-\theta}=\eta_{k-1}  \frac{1}{t_{min}}.$$
Thus, we can conclude that for each $k\geq k_2$ there holds 
$$ e(x_k) \leq \frac{1}{t_{min}} \eta_{k_{un}},$$
where $k_{un}$ represents the index of last  unsuccessful iteration before the iteration $k$. Letting $k\to \infty$ we conclude that $\lim_{k \to \infty} e(x_k)=0$ in this case as well.

Therefore, in all cases we have $ e(x_k) \to 0. $ If $ \tilde{x} $ is an arbitrary accumulation point of $\{x_k\}$ then clearly $ e(\tilde{x}) = 0 $ and hence $ \tilde{x} $ is feasible.  $ \Box $

\begin{lemma}
    \label{lema3}  Assume that Assumption A\ref{A1} holds and let $ \{x_k\}$ be an iterative sequence generated by IPAS algorithm such that $ N_k = N$ for $ k \geq k_2$ and assume that $ \{x_k\}_{k\geq k_2}$ is bounded.
    If there exists $ k_3 \geq k_2$ such that $ k \in K_{su}$ for all $k \geq k_3$ then each accumulation point  of $ \{x_k\} $ is a stationary point of \eqref{problem}. 
\end{lemma}
{\em Proof.} For each $ k \geq k_3$ we have that $ x_{k+1}$ is successful and hence
%, for $ d_k = p_k =d_{\N} $ obtained in Step 2 of IPAS 
we have 
\begin{eqnarray*}
    f(x_{k+1}) & \leq & f(x_k) + c_1 t_k (\nabla f(x_k))^T p_k + \varepsilon_k \\
    & \leq & f(x_k) - c_1 c t_k \|p_k\|^2 + \varepsilon_k\\
    & \leq & f(x_k) - c_1 c t_{\min} \|p_k\|^2 + \varepsilon_k. 
\end{eqnarray*}
Given that $ f $ is  continuous and $\{x_k\}_{k\geq k_2}$ is assumed to be bounded, there must exists a constant $f_{up}$ such that  $ f(x_{k_3}) \leq f_{up}$. Moreover, $f(x_k) \geq f_{low}$ for each $k$  and using the standard arguments we get $\sum_{k=k_3}^{\infty}\|p_k\|^2<\infty$ and $ \lim_{k \to \infty} \|p_k\| = 0.$ Let us denote, as before, the exact projected gradient direction at $ x_k $ by $ d(x_k). $ Then we have 
$$ \|d(x_k) - p_k\| \leq \|A^T(AA^T)^{-1}\|\|r(x_k)\| \leq C_A \eta_k $$ and 
$$ \lim_{k \to \infty}\|d(x_k)\|\leq \lim_{k \to \infty}\|p_k\| + \lim_{k \to \infty}\|d(x_k) - p_k\| \leq \lim_{k\to \infty}\|p_k\| + C_A \eta_k = 0. $$
So, for arbitrary accumulation point $ \tilde{x} = \lim_{k \in K} x_k$ we have $ \|d(\tilde{x})\|= \lim_{k \in K} \|d(x_k)\| = 0. $ By Lemma \ref{lema2} $ \tilde{x} $ is feasible so the statement follows by Lemma \ref{bmr}.  $ \Box $

The following inequality will be used for further analysis. 
\begin{lemma}
    \label{lema4}
    For each $ k \geq k_2$ there holds
  \begin{eqnarray*} & & \nabla^T f(x_k) p_k \leq  -\|p_k\|^2 - p_k^T(\pi_S(y_k)-\tilde{\pi}_S(y_k)) \\ &+& (p_k +\nabla f(x_k)+ \pi_S(y_k)-\tilde{\pi}_S(y_k))^T(\pi_S(x_k)-x_k+\tilde{\pi}_S(y_k)-\pi_S(y_k)). 
  \end{eqnarray*}
\end{lemma}
{\em Proof.} For any $ y \in \mathbb{R}^n $ we have that $ \pi_S(y)$ is a solution of convex problem 
$$ \min_{Az=b} \frac{1}{2}\|z-y\|^2.$$ Denoting $ h(z) =\frac{1}{2}\|z-y\|^2 $ we can state the KKT conditions as follows. A point $ z^* $ is a solution of $ \min_{Az=b} h(z)$ iff $ (\nabla h(z^*))^T(z-z^*) \geq 0$ for all $ z $ such that $ Az=b. $ The gradient of $ h $ is given by $ \nabla h(z) = z-y$ and hence the optimality condition yields  
\begin{equation}
    \label{optKKT}
    (\pi_S(y)-y)^T(z-\pi_S(y)) \geq 0, \mbox{ for all feasible } z.  
\end{equation}Let us now consider $ k \geq k_2$ and take $ y = y_k = x_k - \nabla f(x_k)$ and $ z=\pi_S(x_k)$ in \eqref{optKKT}. We get 
\begin{eqnarray*}
    (\pi_S(y_k) - y_k \pm \tilde{\pi}_S(y_k))^T(\pi_S(x_k) \pm x_k - \pi_S(y_k) \pm \tilde{\pi}_S(y_k) )\geq 0.
\end{eqnarray*}
Recall that $ p_k = \tilde{\pi}_S(y_k) - x_k. $ The previous inequality actually states
$$ (p_k + \nabla f(x_k) + \pi_S(y_k) - \tilde{\pi}_S(y_k))^T (-p_k + \tilde{\pi}_S(y_k) - \pi_S(y_k) + \pi_S(x_k) - x_k) \geq 0 $$
and the statement follows. $ \Box $

\begin{lemma}
    \label{lema5}
    Assume that Assumption A\ref{A1} holds and let $ \{x_k\}$ be an iterative sequence generated by IPAS algorithm such that $ N_k = N$ for $ k \geq k_2$ and assume that $ \{x_k\}_{k\geq k_2}$ is bounded.
     If the sequence of unsuccessful iterations $\{x_k \}_{k\in K_{un}} $ is infinite  then there exists an accumulation point $ \tilde{x}$ of $ \{x_k\} $ such that $ \tilde{x}$ is stationary point for the problem \eqref{problem}. 
\end{lemma}
{\em Proof.} 
First of all let us notice that boundedness of $ \{x_k\}$ together with A\ref{A1} implies that $ \{p_k\}$ is bounded as well. Namely, for $ y_k = x_k-\nabla f(x_k) $ we have 
\begin{eqnarray*}
    p_k & = & \tilde{\pi}_S(y_k) - x_k = y_k - A^T\tilde{\lambda}(y_k) - x_k  \\ 
    & = & - \nabla f(x_k) - A^T((AA^T)^{-1} (Ay_k-b) + (AA^T)^{-1}r(y_k)). 
\end{eqnarray*}
Now, for $C_{A}= \|A^T (AA^T)^{-1}\|$, having $ \|x_k\|\leq C_x,  \|\nabla f(x_k)\| \leq C_g,  $ and therefore $ \|y_k\| \leq \|x_k\| + \|\nabla f(x_k)\| \leq C_x + C_g, $ with  $ \|r(y_k)\|\leq \eta_k $ by conditions of IPAS algorithm we get 
\begin{equation}
    \label{db}
 \|p_k\| \leq C_g + C_A(\|A\|(C_x+C_g)+\|b\|+ \eta_k):=C_p.  \end{equation} 
 
Let us assume that there exists $ \tilde{\varepsilon} > 0 $ such that $ \|p_k\| \geq \tilde{\varepsilon} > 0 $ for all $ k \geq k_2. $
We will consider unsuccessful iterations, i.e. $ k \geq k_2, k \in K_{un}. $ For these iterations we have $ (\nabla f(x_k))^T d_k > -c\|p_k\|^2. $ Lemma \ref{lema4} implies 
\begin{eqnarray} \label{nj1}
    0 & < & -(1-c)\|p_k\|^2 - p_k^T(\pi_S(y_k)-\tilde{\pi}_S(y_k)) \nonumber\\
    & + & (p_k + \nabla f(x_k) + \pi_S(y_k)-\tilde{\pi}_S(y_k))^T(\tilde{\pi}_S(y_k)-\pi_S(y_k)+ \pi_S(x_k)-x_k) \nonumber\\
    & \leq & -(1-c)\varepsilon^2 + \|p_k\|\|\pi_S(y_k)-\tilde{\pi}_S(y_k)\| \nonumber\\
    & + & (\|p_k\| + \|\nabla f(x_k)\| + \|\pi_S(y_k)-\tilde{\pi}_S(y_k))\|)\nonumber\\
    & & \cdot (\|\tilde{\pi}_S(y_k)-\pi_S(y_k)\|+\|\pi_S(x_k)-x_k\|). 
\end{eqnarray}
As 
$$\pi_S(y_k)-\tilde{\pi}_S(y_k) = A^T(AA^T)^{-1}r(y_k),$$ 
we have
\begin{equation}
    \label{nj2}
    \|\tilde{\pi}_S(y_k)-\pi_S(y_k)\|\leq C_A \eta_k. 
\end{equation}
Furthermore, 
\begin{equation}
\label{nj3}
\|\pi_S(x_k)-x_k\| = \|x_k - A^T\lambda(x_k)-x_k\|=\|A^T(AA^T)^{-1}(Ax_k-b)\|\leq C_A e(x_k), 
\end{equation} and $ e(x_k) \to 0  $
by Lemma \ref{lema2}.  Putting together \eqref{db}, \eqref{nj1}-\eqref{nj3} we get 
$$ 0 < -(1-c)\tilde{\varepsilon}^2 + C_p C_A \eta_k + (C_p+C_g+C_A\eta_k)(C_A\eta_k+C_A e(x_k)). $$
Taking the limit for $ k \in K_{un}, k \to \infty$ in the above inequality we get 
$$ 0 \leq -(1-c) \tilde{\varepsilon}^2$$ which can not be true. Thus there is no $ \tilde{\varepsilon} > 0 $ such that $ \|p_k\| \geq \tilde{\varepsilon} > 0 $ for all $ k \geq k_2. $ Therefore there exists  an infinite $ K \subset \mathbb{N}$ such that $ \lim_{k \in K} p_k=0  $ and therefore $ \lim_{k \in K} d(x_k) = 0,$ as in the proof of Lemma \ref{lema3}. As $ \{x_k\}_{k\geq k_2} $ is bounded then there exists $ \tilde{K} \subset K $ such that $ \lim_{k \in \tilde{K}} x_k = \tilde{x} $ and thus $ d(\tilde{x}) = 0. $ By Lemma \ref{lema2} $ \tilde{x} \in S $ and hence    the statement follows by Lemma \ref{bmr}.

\begin{theorem}
    \label{tmain}
    Let Assumption A\ref{A1} holds and assume that $ \{x_k\} $ generated by IPAS algorithm  is bounded. Then a.s. there exists an accumulation point of $ \{x_k\} $ which is a stationary point of \eqref{problem}. 
\end{theorem}
{\em Proof. } If $ N_k < N $ for all $ k $ we have that  the sequence $ \{x_k\}$ is bounded so Assumption A\ref{assNEW} holds and there exists at least one accumulation point of $ \{x_k\}. $ That point is stationary by Corollary \ref{das-mini}. In the case of $ N_k = N $ for $ k $ large enough the statement follows by Lemma \ref{lema3} and Lemma \ref{lema5}. $ \Box$

\section{Numerical results}
In this section we  demonstrate the efficiency of the proposed algorithm through a series of ML experiments. We apply IPAS algorithm to solve equality constrained logistic regression problems.
To evaluate IPAS, we compare its performance with two other notable methods -  Stochastic Sequential Quadratic Programming method (Stochastic SQP) \cite{kapica}, and ASPEN \cite{asp}. Both of these algorithms are designed to deal with equality constrained optimization problems, making them suitable benchmarks for comparison. Additionally, in order to provide better insights into the  proposed method's behavior, we compare four different versions of IPAS obtained by specifying some of the key parameters, namely $\eta_k$ and $\varepsilon_k$, and the sample size update. 

The considered logistic regression problems are given by   
\begin{equation}
\label{logreg}
    \min_{Ax=b}f(x):=\frac{1}{N}\sum_{i=1}^N\log(1+e^{(-y_i(x^Tz_i))}),
\end{equation}
where $N$ is the number of samples, $z_i\in\mathbb{R}^n$ are sample attributes, $y_i\in\{-1,1\}$ are respective labels, and $A\in\mathbb{R}^{m \times n}$,  $b\in\mathbb{R}^m$ define the feasible set.  The equality constraints are simulated to provide a full ranked matrix $A$ with    $m=2n/3$ rows.  The algorithm has been implemented and evaluated on 8 different datasets from LIBSVM repository \cite{SPLiADL}. Each set has been divided into training and test set, in the 4:1 ratio, and the details can be seen in the following table. 

\begin{table}[h!] 
\centering
\begin{tabular}{|l|l|l|l|}
\hline
         & $N_{train}$ & $N_{test}$ & $n$    \\ \hline
MUSHROOM & 6499         & 1625        & 112  \\ \hline
A9A      & 19507        & 4877        & 123  \\ \hline
W7A      & 19754        & 4938        & 300  \\ \hline
MNIST    & 9419         & 2355        & 780  \\ \hline
EPSILON  & 16000        & 4000        & 2000 \\ \hline
CIFAR  & 8000         & 2000        & 3072 \\ \hline
SVHN     & 19557        & 4889        & 3072 \\ \hline
GISETTE  & 4800         & 1200        & 5000 \\ \hline
\end{tabular}
\caption{Classification dataset details}
\label{tab1}
\end{table}
Performance analysis is conducted by  showing the algorithms'  optimality gap against  computational cost measure on training set (Figures \ref{fig1}-\ref{fig8}), and the accuracy measure  tracked over both training and test set (Figure 9). More precisely, the optimality measure is given by $\|x_k-x^*\|_2$, where $x^*$ is obtained by using Matlab \textit{fmincon} function with optimality criterion $10^{-4}$. The computational cost is modeled by number of scalar products needed to reach iteration $k$ and it is greatly influenced by the sample size behavior.  When calculating the number of scalar products, we also take into account the cost of performing inexact projections. We use the Conjugate Gradient method to solve the linear systems inexactly at each iteration, therefore,  every time we use the solver, we account for $(m+4)\cdot iter$ number of scalar products, where $iter$ is the number of iterations Conjugate Gradient method made to achieve the tolerance $\eta_k$.  
It is important to emphasize that we are not relying on the "optimal" implementation, but  rather we are counting the computational cost that is theoretically derived from the algorithm. This is also the case for ASPEN and STO-SQP implementation, thus a fair comparison of all the considered methods is provided.

The initial point  $x_0$ was fixed for all the tested methods and generated by means of standard Gaussian distribution.  For all  tested variants of IPAS method we use additional sample size equal to 1, i.e., $\vert \D_k \vert=1$, and the initial subsampling size $N_0=\lceil 0.01N\rceil$. Furthermore, we set $c_1=  10^{-4}, \beta=0.8, c= 10^{-4}$, and  $C=1$.  For the descent conditions \eqref{ls} and \eqref{ac}, we set the relaxation parameter $\varepsilon_k=(\frac{1}{k^{0.51}})^2$ which satisfies theoretical  recommendations. STO-SQP and ASPEN are implemented with the configurations recommended in \cite{kapica},\cite{asp} respectively.  

Basic IPAS algorithm uses $\eta_k=\frac{1}{k^{0.51}}$ to control inexact projections, which yields relatively slow decay of infeasibility.   When invoked, the increase of the sample size is done by $N_{k+1}=N_k+1$.  
In order to demonstrate the efficacy of inexact projections, we compare IPAS to its version named EXACT which makes uniformly  accurate projections with tolerance $\eta_k=10^{-6}$, while all the other parameters are the same as in basic IPAS.  On the other hand, we also test IPAS-R - a version of IPAS which allows relatively big infeasibility by setting  $\eta_k=10000/k^{0.51}$, while keeps the same configuration of the remaining parameters as IPAS.  
We also examine the influence of the sample size scheduling by testing two more versions of IPAS -  IPAS-M and IPAS-H  -  which update the subsampling size by  $N_{k+1}=1.01N_k$ and $N_{k+1}=1.1N_k$, respectively. All the other parameters are set as for basic IPAS, including $\eta_k$. 
 The following table summarizes the choices of the relevant parameters for tested  IPAS versions.

\begin{table}[h!]
\centering
\begin{tabular}{|l|l|l|l|}
\hline
       & $\eta_k$         & $\varepsilon_k$ & $N_{k+1}$ \\ \hline
IPAS   & $1/k^{0.51}$     & $(1/k^{0.51})^2$    & $N_k+1$   \\ \hline
IPAS-R & $10000/k^{0.51}$ & $(1/k^{0.51})^2$    & $N_k+1$   \\ \hline
EXACT  & $10^{-6}$        & $(1/k^{0.51})^2$    & $N_k+1$   \\ \hline
IPAS-M   & $1/k^{0.51}$     & $(1/k^{0.51})^2$    & $1.01N_k$   \\ \hline
IPAS-H   & $1/k^{0.51}$     & $(1/k^{0.51})^2$    & $1.1N_k$   \\ \hline
\end{tabular}
\caption{IPAS versions' parameters}
\label{tab2}
\end{table}

Figures \ref{fig1}-\ref{fig8} show the behavior of all the consider algorithms applied on problem (\ref{logreg}) with datasets from Table \ref{tab1}. The graphs  on the left show the optimality gap with respect to computational costs (scalar products), on the training set. The graphs in the middle show the portion of the full training sample size used in each iteration, while the graphs on the right track infeasibility measure $e(x_k)=\|Ax_k-b\|_2$ in terms of computational cost.  Notice that none of the IPAS configurations reaches full sample size for the duration of the simulation.  

In Figures \ref{fig1}, \ref{fig4}, \ref{fig5} and \ref{fig8} it can be seen that IPAS-R is closest to the optimal point for the given budget. In Figures \ref{fig2}, \ref{fig3}, \ref{fig6} and \ref{fig7} IPAS-R showed worse results, implying that the over-relaxation was not a good choice in these examples. This is most likely due to heterogeneity of the datasets. Nonetheless, basic IPAS provided  consistently good results overall. It outperforms ASPEN and STO-SQP in almost all the considered examples. ASPEN is farthest from optimality in most cases, which might be the consequence of high cost of evaluating the penalization parameter for every function evaluation and heterogeneity. This behavior is expected given that ASPEN is constructed for more general  problems with nonlinear equality constraints, while IPAS focuses on the linear case. STO-SQP is more competitive and mostly performs better then ASPEN, but it also shows instabilities, especially for EPSILON dataset. Moreover, notice that  IPAS also performs better then EXACT in most of the cases, which implies that inexact projections influenced the reduction of the overall computational costs.  

Considering the sample size behavior, IPAS-H provides the fastest increase as expected, and it is followed by IPAS-M. It is notable that IPAS-H reaches the full sample size in a substantial number of cases, while all the other methods stay well below except for SVHN dataset where IPAS-M also attains the full sample within the given budget. However, this aggressive sample size increase seems to be beneficial in some cases such as for dataset A9A in Figure \ref{fig2}, most probably because of heterogeneous data. Although basic IPAS method showed to be robust and provides a safe choice in general for all the considered problems, it is evident that an the sample size update heavily influences the performance of the algorithm and is is problem dependent. This opens a new research direction for future work. 

Finally,  Figure \ref{grid} shows the training and test accuracies for all the considered datasets. It can be seen that IPAS-R reaches high accuracies, while other IPAS versions also demonstrate large increases in accuracies for the least amount of computational cost for most of the datasets. However, the constraints of the considered problems are simulated and reaching the feasibility may deteriorate the accuracy overall.

\begin{figure}[h!] 
\centering
\includegraphics[width=12cm]{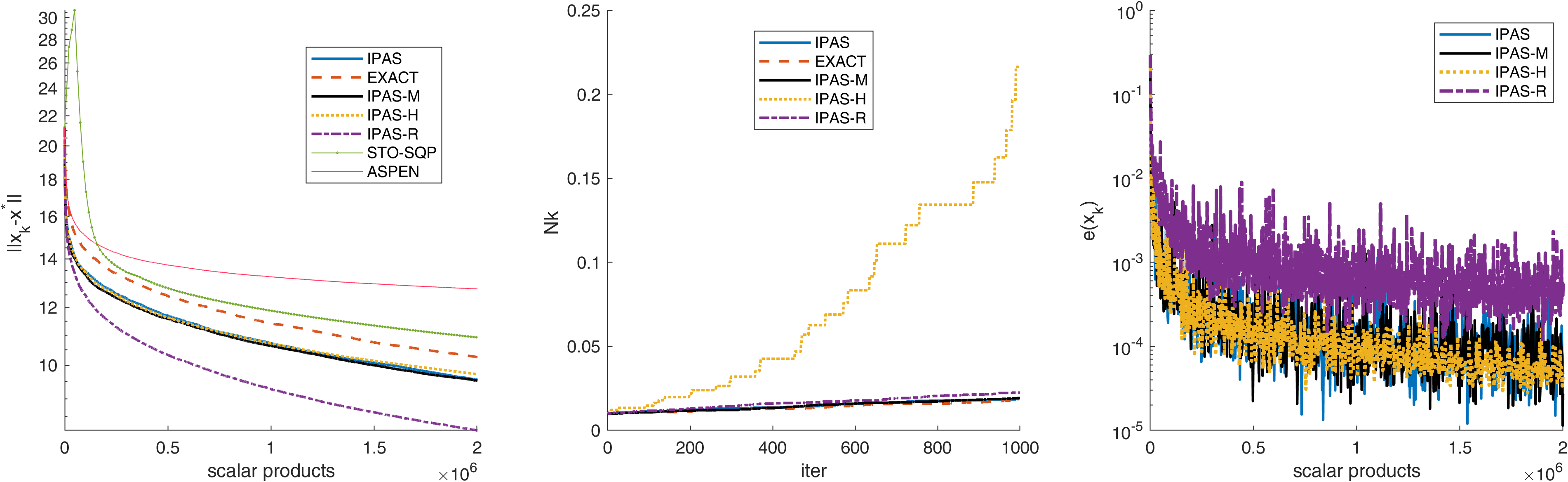} 
 \vspace{2mm}
\begin{tabular}{>{\centering\arraybackslash}m{0.33\textwidth}
                    >{\centering\arraybackslash}m{0.33\textwidth}
                    >{\centering\arraybackslash}m{0.33\textwidth}}
        a) & b) & c)
    \end{tabular}
\caption{{\scriptsize{MUSHROOM dataset, $N=6499, n=112$.  Comparison of algorithms in Table \ref{tab2} with ASPEN and STO-SQP on the training set: a) optimality gap against computational costs; b) sample size behavior in terms of iterations; c) infeasibility measure against computational costs. }}}	
\label{fig1}
\end{figure}

\begin{figure}[h!]
\centering
\includegraphics[width=12cm]{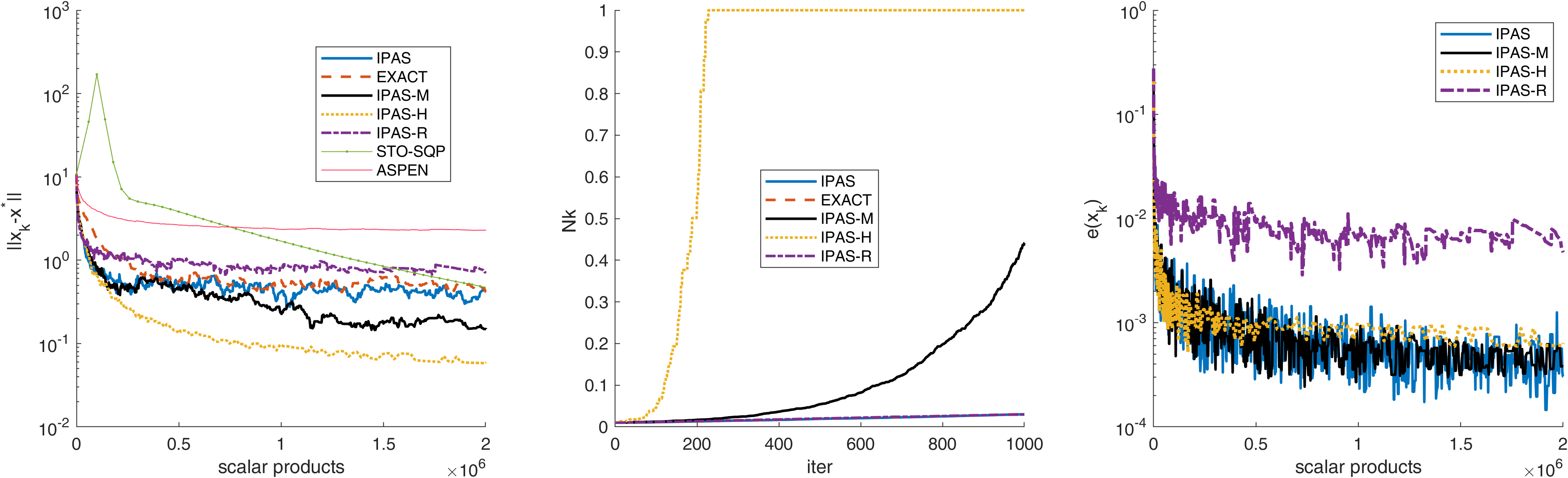} 
 \vspace{2mm}
\begin{tabular}{>{\centering\arraybackslash}m{0.33\textwidth}
                    >{\centering\arraybackslash}m{0.33\textwidth}
                    >{\centering\arraybackslash}m{0.33\textwidth}}
        a) & b) & c)
    \end{tabular}
\caption{{\scriptsize{A9A dataset, $N=19507, n=123$. Comparison of algorithms in Table \ref{tab2} with ASPEN and STO-SQP on the training set: a) optimality gap against computational costs; b) sample size behavior in terms of iterations; c) infeasibility measure against computational costs. }}}	
\label{fig2}
\end{figure}

\begin{figure}[h!]
\centering
\includegraphics[width=12cm]{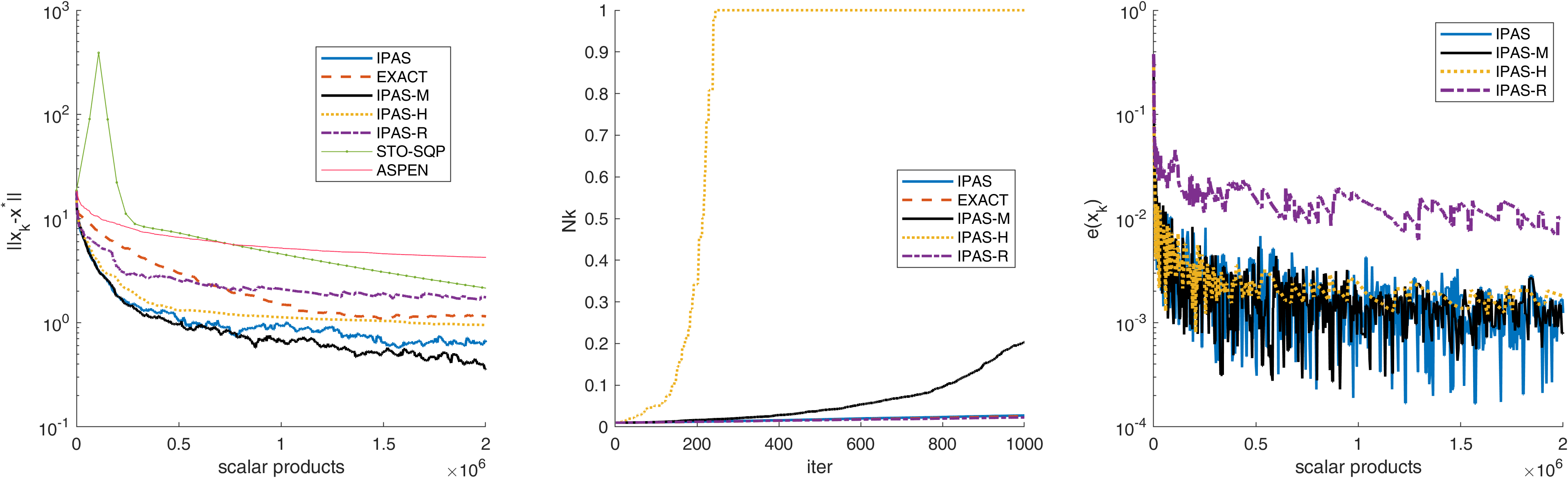} 
 \vspace{2mm}
\begin{tabular}{>{\centering\arraybackslash}m{0.33\textwidth}
                    >{\centering\arraybackslash}m{0.33\textwidth}
                    >{\centering\arraybackslash}m{0.33\textwidth}}
        a) & b) & c)
    \end{tabular}
\caption{{\scriptsize{W7A dataset, $N=19754, n=300$. Comparison of algorithms in Table \ref{tab2} with ASPEN and STO-SQP on the training set: a) optimality gap against computational costs; b) sample size behavior in terms of iterations; c) infeasibility measure against computational costs. }}}	 \label{fig3}
\end{figure}

\begin{figure}[h!]
\centering
\includegraphics[width=12cm]{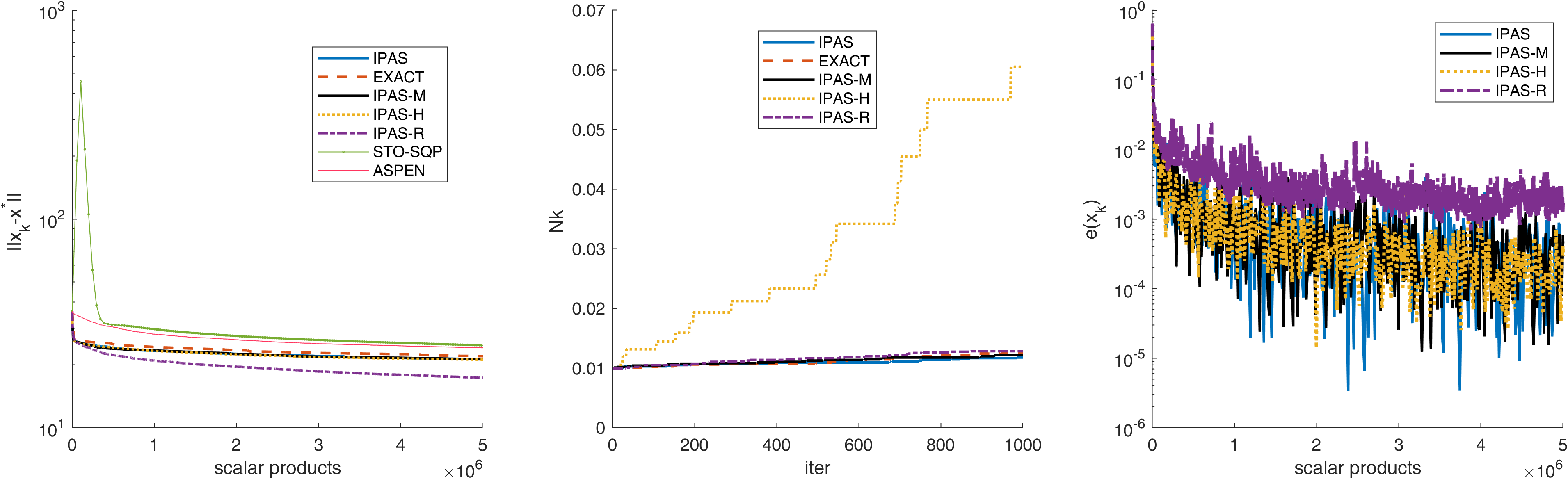} 
 \vspace{2mm}
\begin{tabular}{>{\centering\arraybackslash}m{0.33\textwidth}
                    >{\centering\arraybackslash}m{0.33\textwidth}
                    >{\centering\arraybackslash}m{0.33\textwidth}}
        a) & b) & c)
    \end{tabular}
\caption{{\scriptsize{MNIST dataset, $N=9419, n=780$. Comparison of algorithms in Table \ref{tab2} with ASPEN and STO-SQP on the training set: a) optimality gap against computational costs; b) sample size behavior in terms of iterations; c) infeasibility measure against computational costs. }}}	 \label{fig4}
\end{figure}

\begin{figure}[h!] 
\centering
\includegraphics[width=12cm]{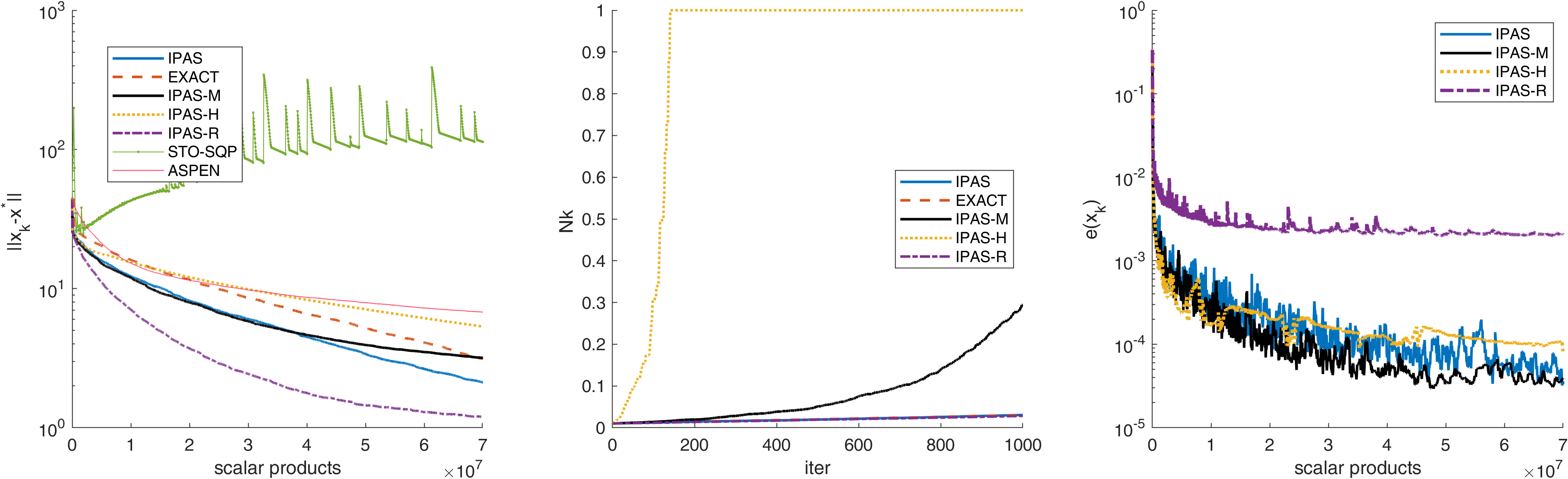} 
 \vspace{2mm}
\begin{tabular}{>{\centering\arraybackslash}m{0.33\textwidth}
                    >{\centering\arraybackslash}m{0.33\textwidth}
                    >{\centering\arraybackslash}m{0.33\textwidth}}
        a) & b) & c)
    \end{tabular}
\caption{{\scriptsize{EPSILON dataset, $N=16000, n=2000$. Comparison of algorithms in Table \ref{tab2} with ASPEN and STO-SQP on the training set: a) optimality gap against computational costs; b) sample size behavior in terms of iterations; c) infeasibility measure against computational costs. }}}	\label{fig5}
\end{figure}

\begin{figure}[h!]
\centering
\includegraphics[width=12cm]{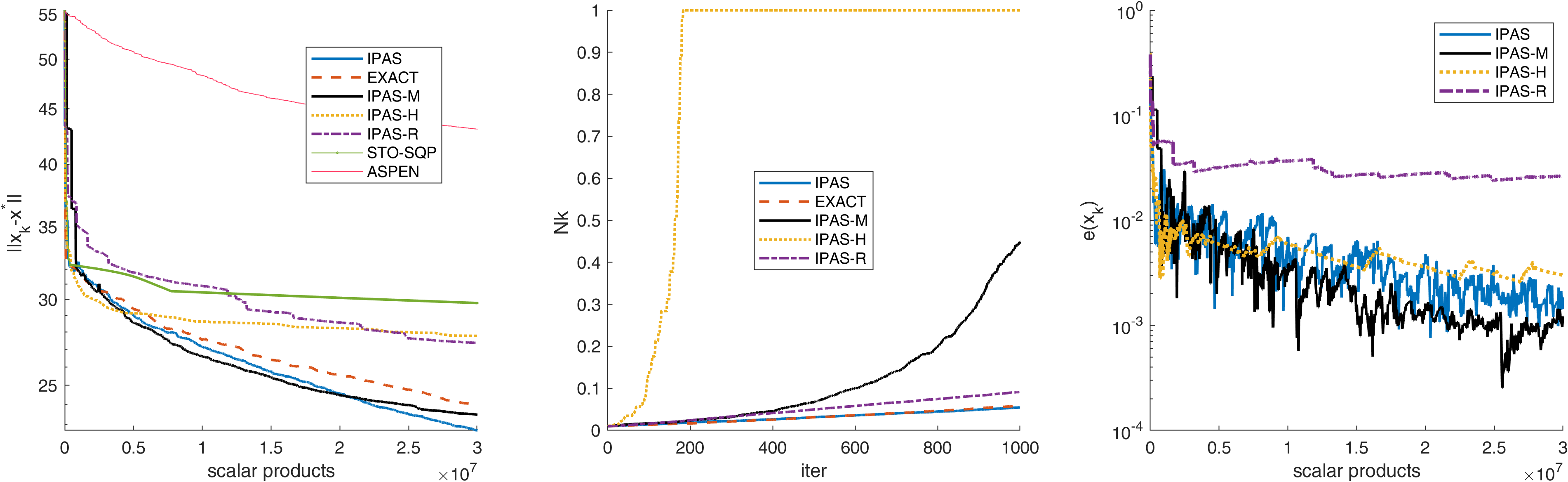} 
 \vspace{2mm}
\begin{tabular}{>{\centering\arraybackslash}m{0.33\textwidth}
                    >{\centering\arraybackslash}m{0.33\textwidth}
                    >{\centering\arraybackslash}m{0.33\textwidth}}
        a) & b) & c)
    \end{tabular}
\caption{{\scriptsize{CIFAR dataset, $N=8000, n=3072$. Comparison of algorithms in Table \ref{tab2} with ASPEN and STO-SQP on the training set: a) optimality gap against computational costs; b) sample size behavior in terms of iterations; c) infeasibility measure against computational costs. }}}	 \label{fig6}
\end{figure}
\newpage 
\begin{figure}[h!] 
\centering
\includegraphics[width=12cm]{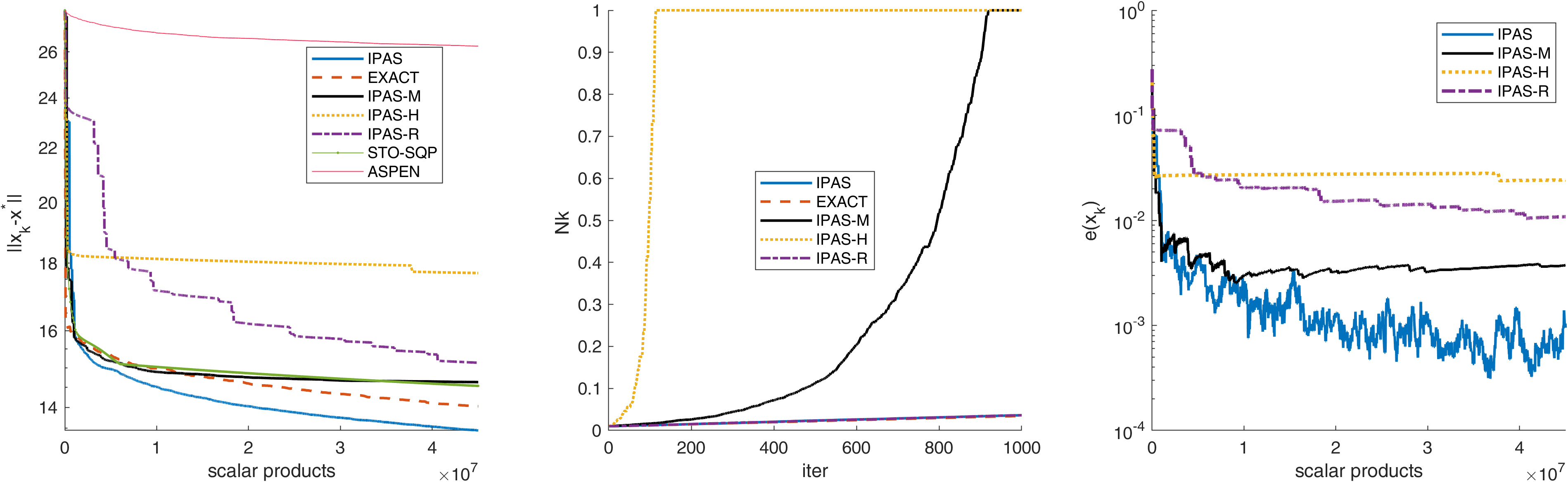} 
 \vspace{2mm}
\begin{tabular}{>{\centering\arraybackslash}m{0.33\textwidth}
                    >{\centering\arraybackslash}m{0.33\textwidth}
                    >{\centering\arraybackslash}m{0.33\textwidth}}
        a) & b) & c)
    \end{tabular}
\caption{{\scriptsize{SVHN dataset, $N=19557, n=3072$. Comparison of algorithms in Table \ref{tab2} with ASPEN and STO-SQP on the training set: a) optimality gap against computational costs; b) sample size behavior in terms of iterations; c) infeasibility measure against computational costs. }}}	\label{fig7}
\end{figure}

\begin{figure}[h!]
\centering
\includegraphics[width=12cm]{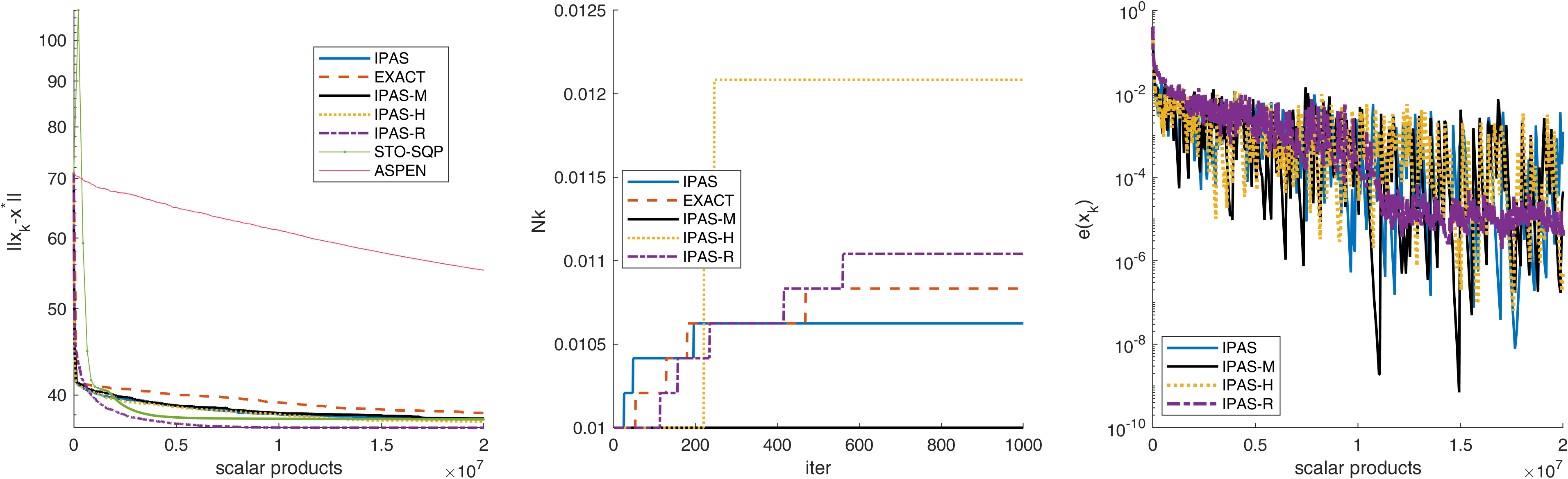} 
 \vspace{2mm}
\begin{tabular}{>{\centering\arraybackslash}m{0.33\textwidth}
                    >{\centering\arraybackslash}m{0.33\textwidth}
                    >{\centering\arraybackslash}m{0.33\textwidth}}
        a) & b) & c)
    \end{tabular}
\caption{{\scriptsize{GISETTE dataset, $N=4800, n=5000$. Comparison of algorithms in Table \ref{tab2} with ASPEN and STO-SQP on the training set: a) optimality gap against computational costs; b) sample size behavior in terms of iterations; c) infeasibility measure against computational costs. }}}	 \label{fig8}
\end{figure}

\begin{figure}[htbp]
\centering

\begin{tabular}{cc}
\includegraphics[width=0.49\textwidth]{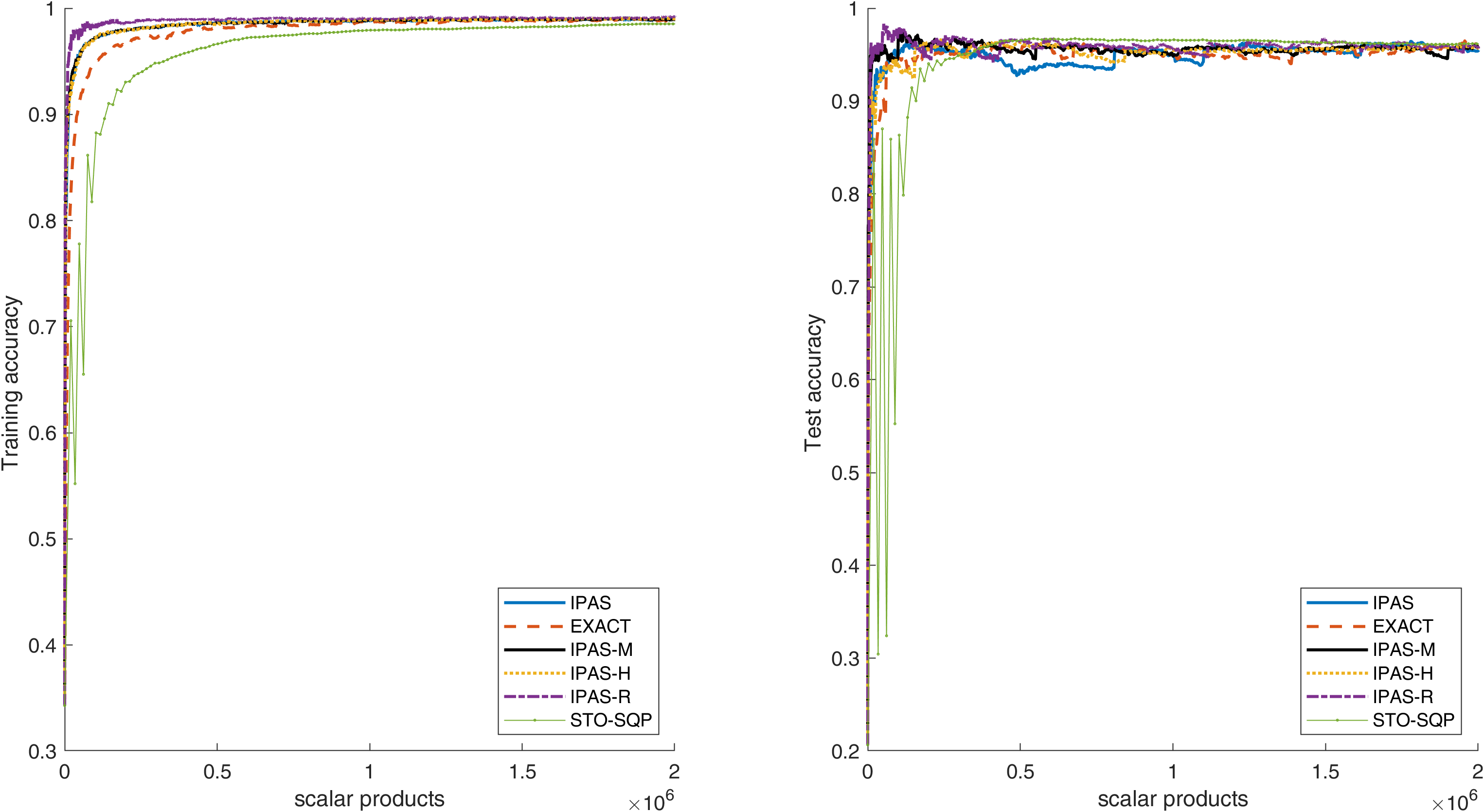} &
\includegraphics[width=0.49\textwidth]{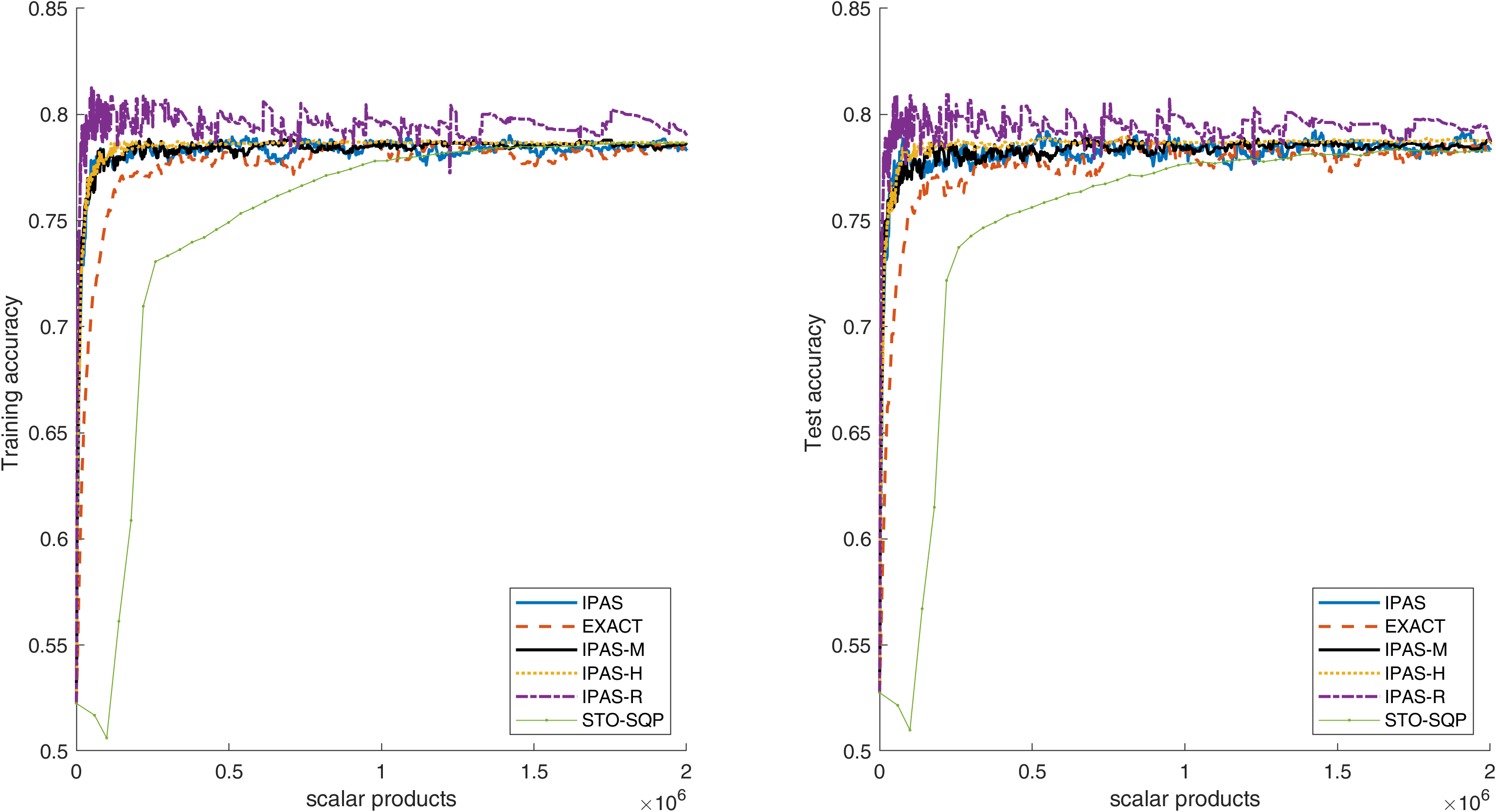} \\
a) & b) \\[3mm]

\includegraphics[width=0.49\textwidth]{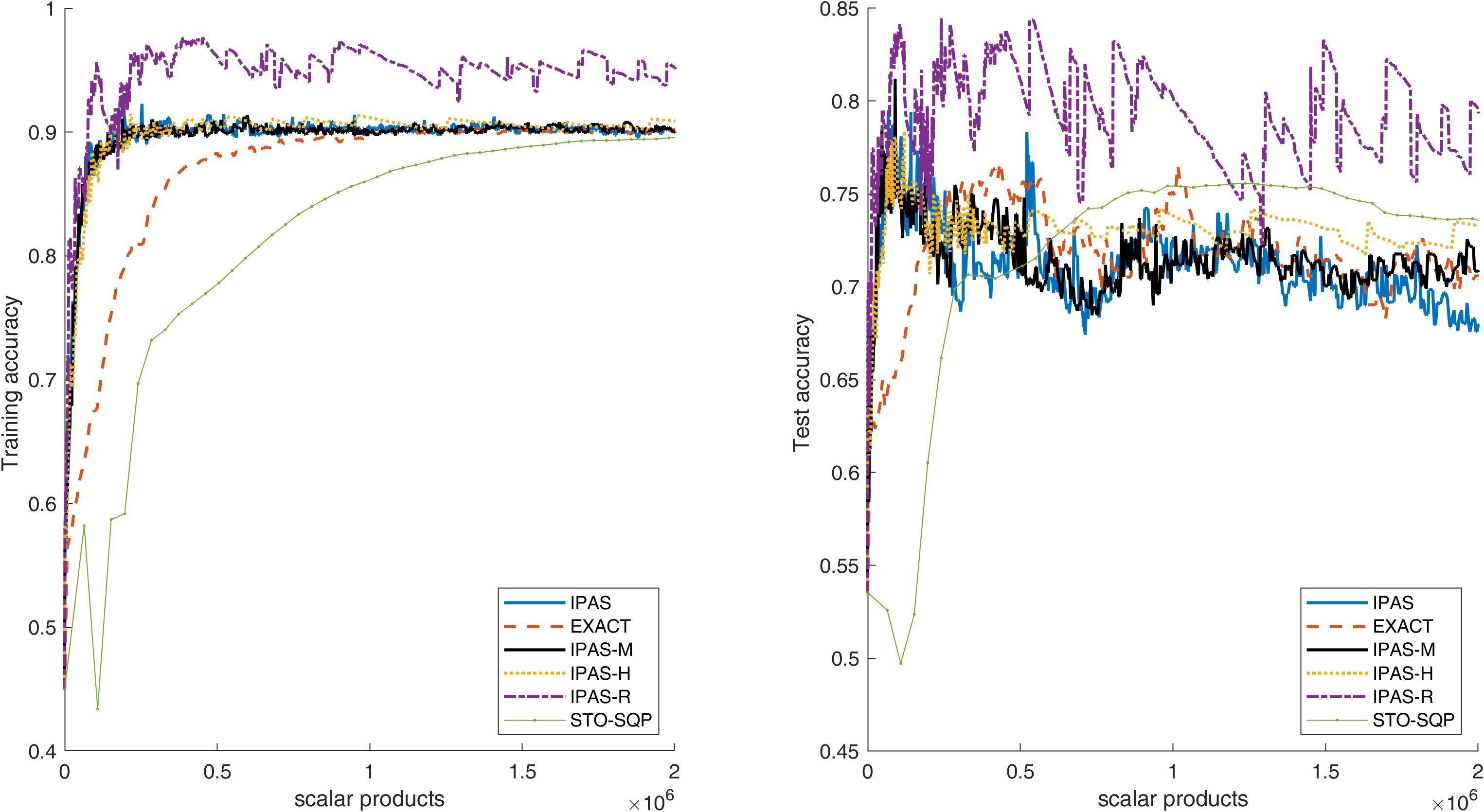} &
\includegraphics[width=0.49\textwidth]{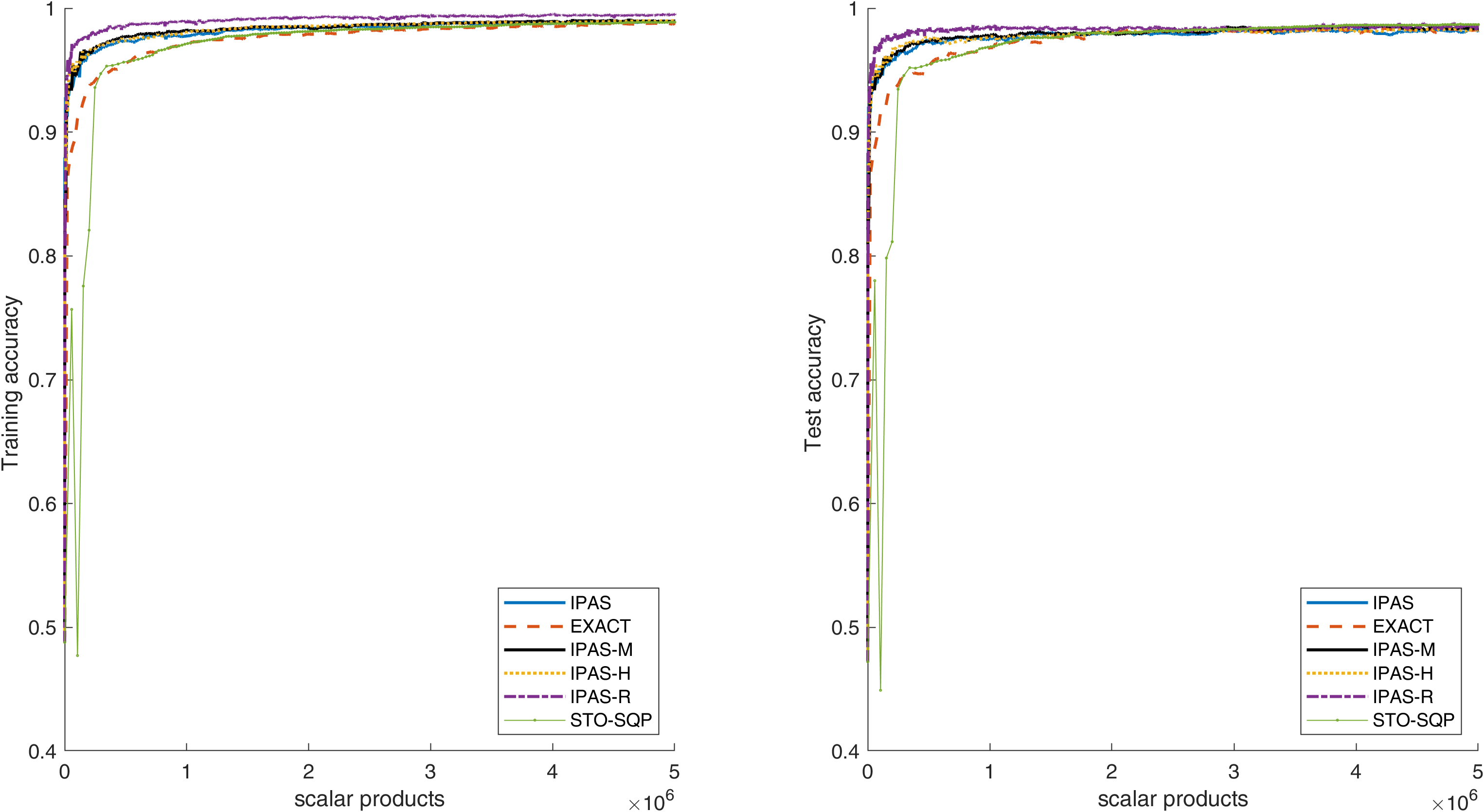} \\
c) & d) \\[3mm]

\includegraphics[width=0.49\textwidth]{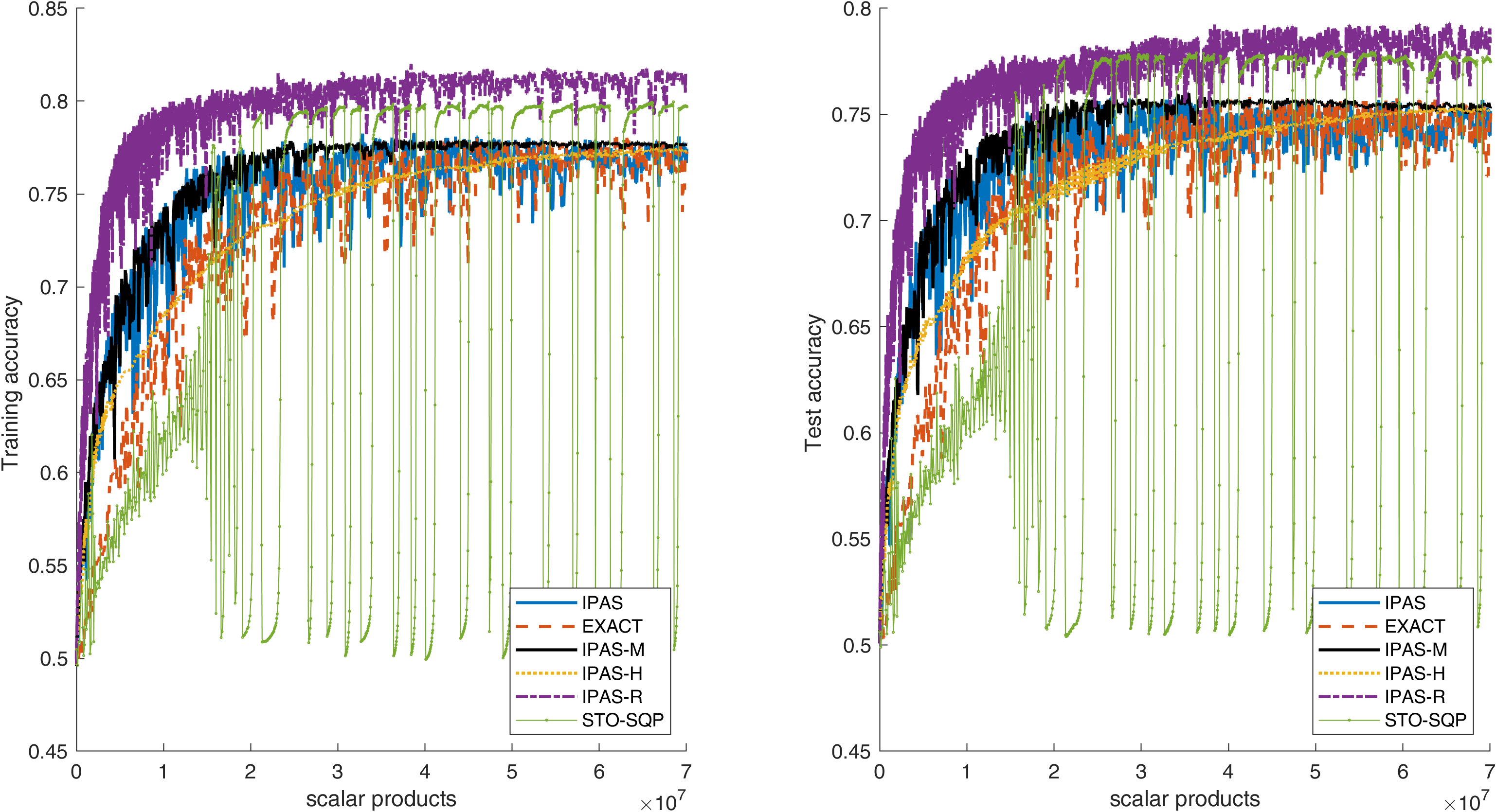} &
\includegraphics[width=0.49\textwidth]{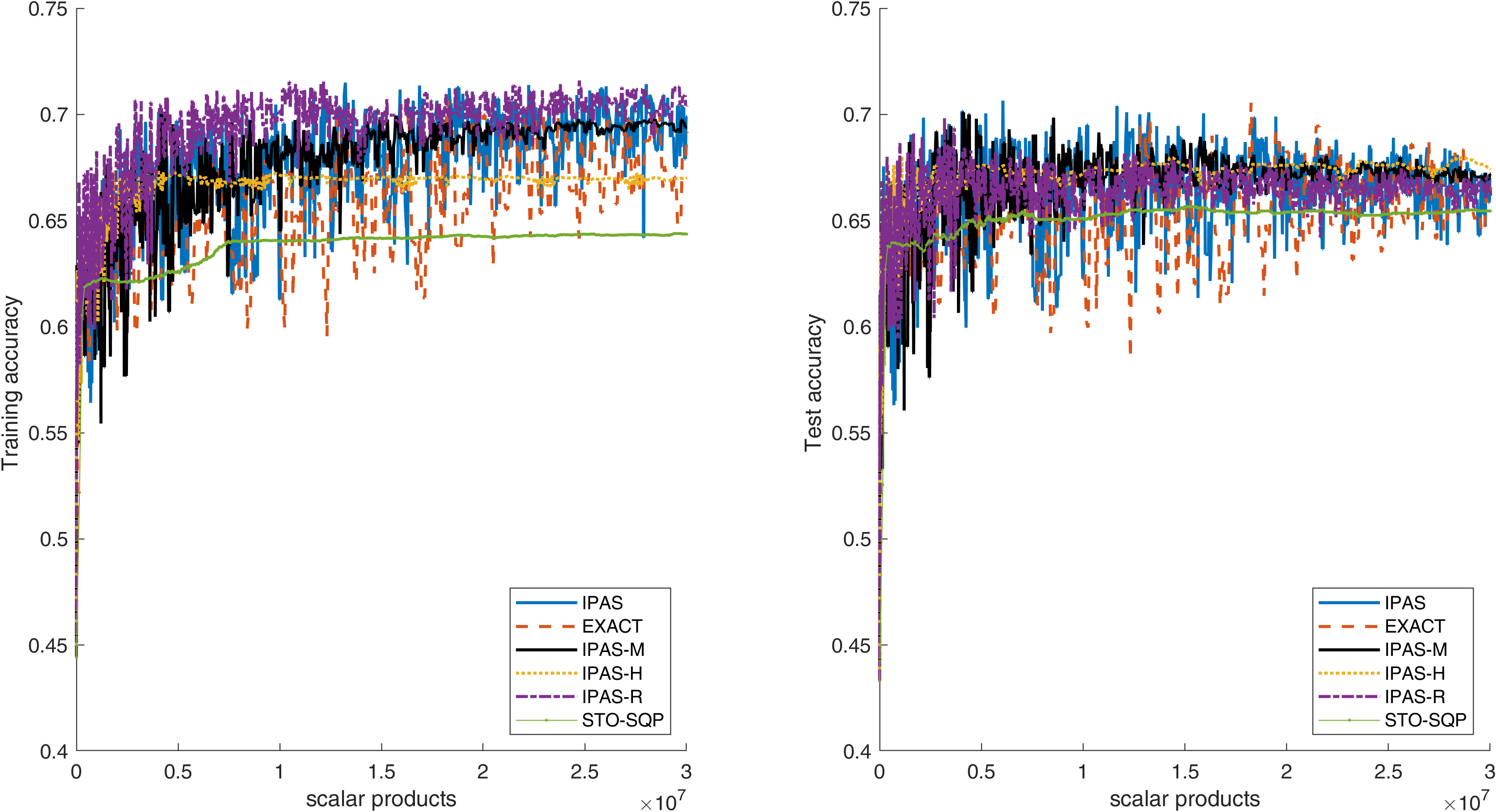} \\
e) & f) \\[3mm]

\includegraphics[width=0.49\textwidth]{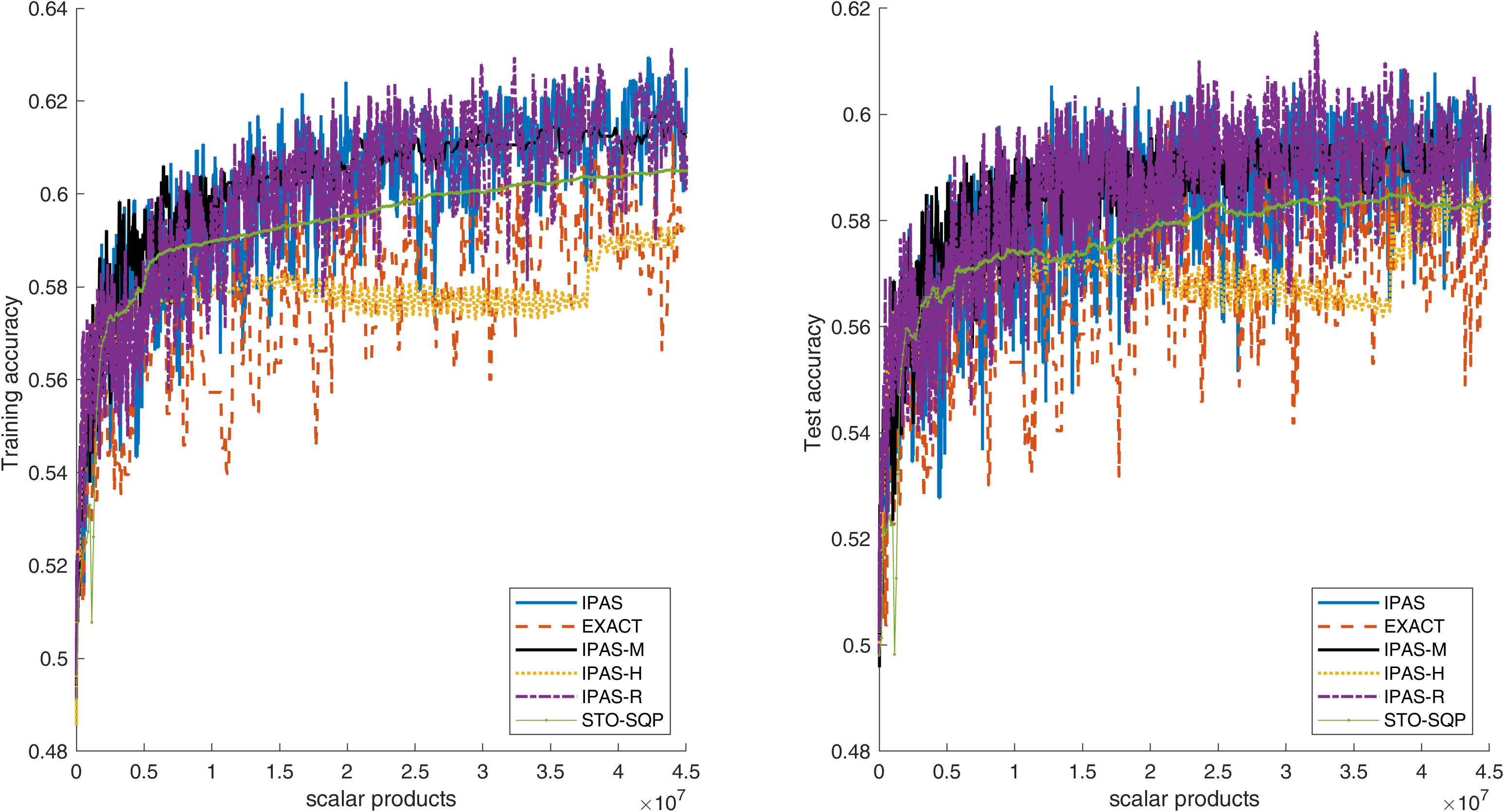} &
\includegraphics[width=0.49\textwidth]{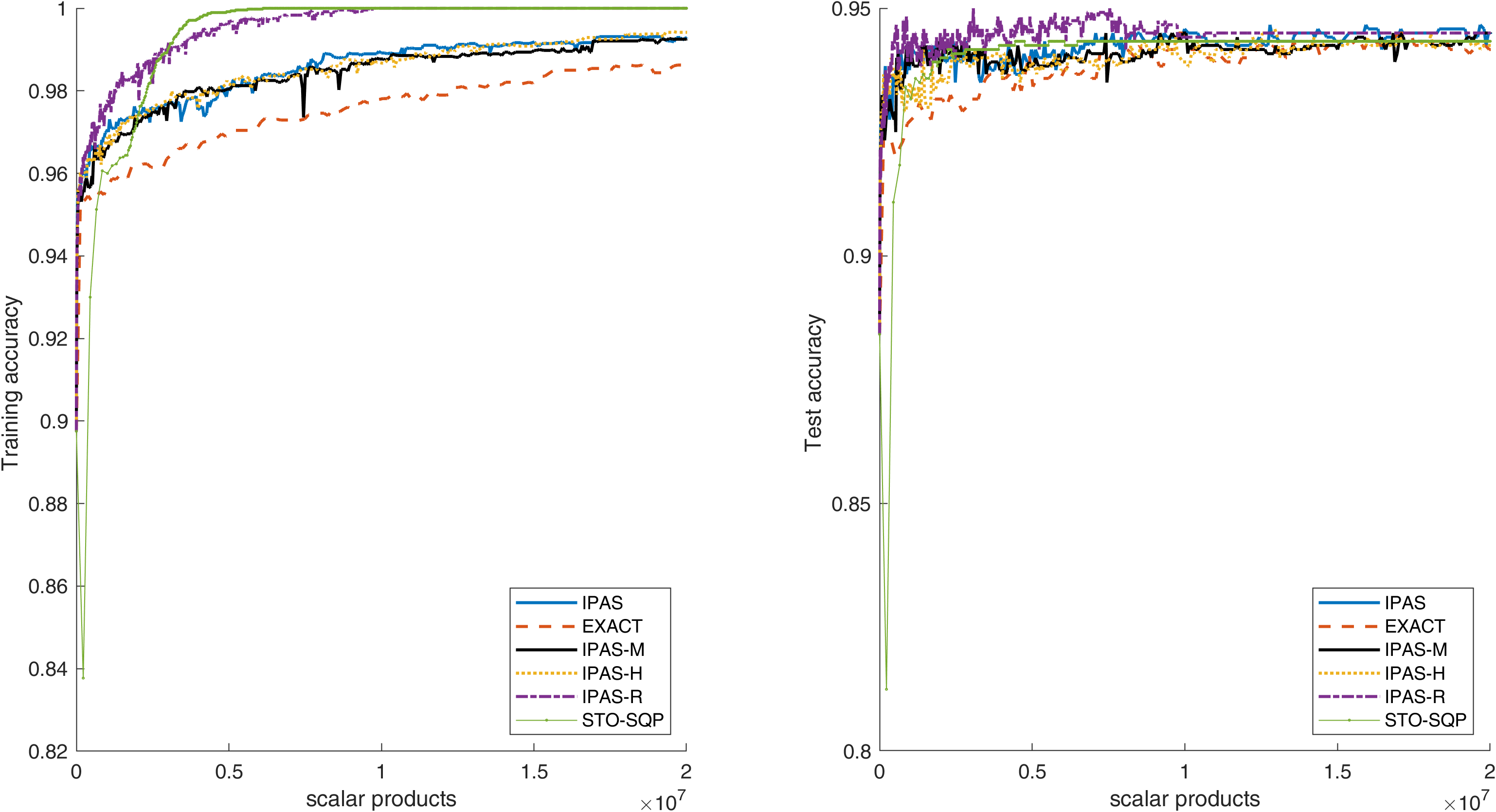} \\
g) & h)
\end{tabular}

\caption{\scriptsize{Model evaluation comparison for different datasets. Training error on the left, and test error on the right. Datasets: a) MUSHROOM ; b) A9A ; c) W7A ; d) MNIST ; e) EPSILON ; f) CIFAR ; g) SVHN ; h) GISETTE} }
\label{grid}
\end{figure}

\newpage

\section{Conclusions}
The proposed algorithm  represents a novel approach for solving weighted sum problems with possibly large number of linear equality constraints. It adapts  additional sampling approach originally constructed for finite sum unconstrained problems to  more general problems of the form \eqref{problem}. 
Moreover, inexact projections are allowed, but controlled by a predefined sequence of parameters. Allowing inexact projections  shows to be very important in terms of computational costs, especially when the number of constraints is large. The almost sure convergence of the proposed method is proved under a set of  standard assumptions for the stochastic framework, without the convexity assumption.   Preliminary numerical results on  a number of machine learning problems with real-world data and  simulated constraints show that IPAS is competitive with the relevant benchmark methods in this field.  Possible future work may include fine-tuning for some special subclasses of the considered problems and finding an optimal sample size increase strategy within the proposed framework.

\section{Funding} 

This research was supported by the Science Fund of the Republic of Serbia, GRANT No 7359, Project title - LASCADO.

%\section{Appendix}

\end{document}